\documentclass[11pt]{amsart}
\usepackage{mathrsfs,comment}
\usepackage[usenames,dvipsnames]{color}
\usepackage[normalem]{ulem}
\usepackage{url}
\usepackage[all,arc,2cell]{xy}
\UseAllTwocells
\usepackage{enumerate}
\usepackage{hyperref}
\hypersetup{%
bookmarksnumbered=true,%
colorlinks=true,%
linkcolor=blue,%
citecolor=blue,%
filecolor=blue,%
menucolor=blue,%
urlcolor=blue,%
pdfnewwindow=true,%
pdfstartview=FitBH}

\newif\ifdebug
\debugfalse

\newif \iffig
\figfalse

\newif \iftable
\tabletrue

\usepackage{sfilip_abbreviations}
\usepackage{sfilip_thm_style_long}
\usepackage{sfilip_package_settings}

\def\makeautorefname#1#2{\expandafter\def\csname#1autorefname\endcsname{#2}}
\def\equationautorefname~#1\null{(#1)\null}
\makeautorefname{footnote}{footnote}%
\makeautorefname{item}{item}%
\makeautorefname{figure}{Figure}%
\makeautorefname{table}{Table}%
\makeautorefname{part}{Part}%
\makeautorefname{appendix}{Appendix}%
\makeautorefname{chapter}{Chapter}%
\makeautorefname{theorem}{Theorem}%
\makeautorefname{thm}{Theorem}%
\makeautorefname{cor}{Corollary}%
\makeautorefname{lem}{Lemma}%
\makeautorefname{prop}{Proposition}%
\makeautorefname{pro}{Property}
\makeautorefname{conj}{Conjecture}%
\makeautorefname{defn}{Definition}%
\makeautorefname{notn}{Notation}
\makeautorefname{notns}{Notations}
\makeautorefname{rem}{Remark}%
\makeautorefname{quest}{Question}%
\makeautorefname{exmp}{Example}%
\makeautorefname{ax}{Axiom}%
\makeautorefname{claim}{Claim}%
\makeautorefname{ass}{Assumption}%
\makeautorefname{asss}{Assumptions}%
\makeautorefname{con}{Construction}%
\makeautorefname{prob}{Problem}%
\makeautorefname{warn}{Warning}%
\makeautorefname{obs}{Observation}%
\makeautorefname{conv}{Convention}%

\theoremstyle{definition}

\makeatletter
\let\c@obs=\c@thm
\let\c@cor=\c@thm
\let\c@prop=\c@thm
\let\c@lem=\c@thm
\let\c@prob=\c@thm
\let\c@con=\c@thm
\let\c@conj=\c@thm
\let\c@defn=\c@thm
\let\c@notn=\c@thm
\let\c@notns=\c@thm
\let\c@exmp=\c@thm
\let\c@ax=\c@thm
\let\c@pro=\c@thm
\let\c@ass=\c@thm
\let\c@warn=\c@thm
\let\c@rem=\c@thm
\let\c@sch=\c@thm
\let\c@equation\c@thm
\makeatother

\title{Finiteness of totally geodesic hypersurfaces}

\thanks{Revised \textsc{\today}}

\date{June 2024}

\author{Simion Filip}
\address{Department of Mathematics\\University of Chicago\\Chicago, IL 60637}
\email{sfilip@math.uchicago.edu}

\author{David Fisher}
\address{Department of Mathematics\\Rice University\\Houston, TX 77005}
\email{df32@rice.edu}

\author{Ben Lowe}
\address{Department of Mathematics\\University of Chicago\\Chicago, IL 60637}
\email{loweb24@uchicago.edu}

\begin{document}

\begin{abstract}
We prove that a closed negatively curved analytic Riemannian manifold that contains infinitely many totally geodesic hypersurfaces is isometric to an arithmetic hyperbolic manifold.
Equivalently, any closed analytic Riemannian manifold with negative sectional curvature has only finitely many totally geodesic hypersurfaces, unless it has constant curvature.
\end{abstract}

\maketitle


\section{Introduction}
	\label{sec:introduction}

The main result of this paper is:
\begin{theoremintro}
	\label{thm:main}
	Let $(M,g)$ be a closed, real-analytic Riemannian manifold with negative sectional curvature, of dimension $n\geq 3$.
	Suppose that $M$ contains infinitely many closed totally geodesic immersed hypersurfaces.

	Then $M$ is isometric to a hyperbolic manifold.
	Furthermore $M$ is arithmetic.
\end{theoremintro}
\noindent In the above statement, ``hyperbolic'' refers to a Riemannian manifold with constant strictly negative sectional curvature.
Equivalently, up to rescaling the Riemannian metric, $M$ is isometric to $\bH^n/\Gamma$ where $\bH^n$ is hyperbolic $n$-space, and $\Gamma$ is a lattice which is moreover arithmetic (see \cite[Ch.~IX]{Margulis_Discrete-subgroups-of-semisimple-Lie-groups1991} for the notion of arithmetic lattice).

\subsubsection*{Remarks}
\leavevmode
\begin{enumerate}
\item Our proof works under the more general condition that the geodesic flow of $(M,g)$ is Anosov, see e.g. \cite[Thm.~3.2]{Eberlein1973_When-is-a-geodesic-flow-of-Anosov-type-III} for equivalent characterizations of this property.
Our proof also applies to $M$ of finite volume instead of compact, with some additional technical assumptions spelled out in \autoref{sssec:assumptions_for_stable_unstable_manifolds} and \autoref{sssec:recurrence_in_the_noncompact_case}.

\item The arithmeticity of $M$ is a consequence of the prior conclusions and \cite[Thm.~1.1]{BFMS1}. See also \cite[Thm.~1.1]{MohammadiMargulis2022_Arithmeticity-of-hyperbolic-3-manifolds-containing-infinitely-many-totally} for the case $n=3$.

\item Arithmetic hyperbolic manifolds of the first type have infinitely many closed totally geodesic immersed hypersurfaces.  All other arithmetic hyperbolic manifolds do not. See \cite[Corollary 5.12]{belolipetsky2021subspace} and the remark following \cite[Theorem 1.2]{belolipetsky2021subspace}.

\item It would be interesting to weaken the assumption of analyticity to smoothness, or relax the Anosov assumption to the manifold being of rank $1$ and nonpositively curved (see \cite{knieper98} for this last condition).
In work in preparation, we will address the rank $1$ case, where without analyticity of the metric the result does not hold.
See also \cite[Lecture~III]{BallmannGromovSchroeder1985_Manifolds-of-nonpositive-curvature} for other results that hold in the analytic but not smooth setting of nonpositively curved manifolds.

Additionally, it is worth noting that all other currently available finiteness results that motivate this work are in the homogeneous, or complex-analytic and algebraic settings, which are substantially more constrained than our case.
For example, they only apply to geometric settings that only depend on finite-dimensional parameters.
\end{enumerate}

The motivation for this result can be seen as coming from three different directions: classical Riemannian geometry, a growing body of finiteness results in algebraic and differential geometry and dynamics and an emerging literature integrating ideas from homogeneous dynamics into more general smooth dynamical contexts.

Some of the most sought after results in classical Riemannian geometry provide conditions under which
metrics have constant curvature.
\autoref{thm:main} is far from the first result to give such a characterization in terms of totally geodesic submanifolds. In 1928 Cartan proved his ``axiom of $k$-planes''.

\begin{theoremintro}[Cartan]
\label{thm:Cartan}
Let $1<k<n$ and let $(M,g)$ be a Riemannian manifold. Assume that for every $x$ in $M$ and every $k$ plane
$V$ in $TM_x$, there is a totally geodesic manifold $N \subset M$ through $x$ such that $TN_x=V$.  Then
$M$ is a space form.
\end{theoremintro}

\noindent See \cite{Cartan} and \cite[Thm.~1.8]{Dajczer1990_Submanifolds-and-isometric-immersions} for modern presentations of the proof.  The axiom of planes, which requires that through any three points in a Riemannian manifold passes a totally geodesic surface, was originally introduced by Riemann \cite{riemann1854uber}.

Theorem \ref{thm:Cartan} is not only a significant motivation for our work but also a key step in our proof. We note that
in this context there is no version of \autoref{thm:main} for manifolds of positive curvature or
for general non-compact manifolds, see \autoref{ssec:examples}.
We also remark that there is always
an open dense set of metrics on a manifold for which there are no totally geodesic submanifolds of dimension
greater than $1$ \cite{ElHWilhelm, MurphyWilhelm, Spivak} and \cite{LytchakPetrunin} for an even stronger result for generic metrics.
Numerous variants of Cartan's result characterize other special
metrics in terms of the presence of large families of special submanifolds \cite{survey_tot_geod_submanifolds}.

The second set of motivations for our work is a broad body of finiteness results which show, under various circumstances, that manifolds or varieties contain only finitely many ``special'' submanifolds or subvarieties.
To frame our results in this context,
we restate \autoref{thm:main} in its equivalent contrapositive form:

\begin{theoremintro}
\label{thm:finite}
Let $(M,g)$ be a closed, real-analytic Riemannian manifold with non-constant negative sectional curvature of dimension $n\geq 3$.
Then $M$ contains only finitely many closed totally geodesic immersed hypersurfaces.
\end{theoremintro}

\noindent Results of this kind can be found for example in \Teichmuller dynamics \cite{EskinFilipWright2018_The-algebraic-hull-of-the-Kontsevich-Zorich-cocycle}, in the geometry of symmetric spaces and homogeneous dynamics \cite{BFMS1, BFMS2, BaldiUllmo, LeeOh, MMO1, MMO2, MohammadiMargulis2022_Arithmeticity-of-hyperbolic-3-manifolds-containing-infinitely-many-totally}, and in algebraic geometry and the theory of variations of Hodge structure \cite{BaldiKlinglerUllmo}.  We remark also that the third author obtained a result similar to \autoref{thm:finite} but under strong assumptions on the topology and curvature of $(M,g)$ \cite{Lowe} (see also \cite{cmn22}.)  Studying conditions for finiteness of totally geodesic manifolds in general seems natural in the context of these prior works, see \autoref{ssec:directions_for_future_work_and_conjectures} for a conjecture generalizing \autoref{thm:main} with this context in mind. In Proposition \ref{prop:exactlyk} we construct, for each dimension $n\geq 3$ and each positive integer $k$,  non-hyperbolic negatively curved analytic metrics $g_k$ on a fixed closed manifold $M$ that have exactly $k$ totally geodesic hypersurfaces.

The final motivation and context for our result comes from a growing literature that extends both results and techniques from homogeneous dynamics to more general smooth dynamical settings, see \cite[Section 1]{Fisher-ICM} for a brief discussion of this in the context of the study of invariant measures.  A novel aspect of our proofs in this context are that they are mainly topological and not ergodic theoretic.  This enables us to also prove:

\begin{theoremintro}
\label{thm:unclosed}
Let $n$ be at least $3$, not divisible by $4$ and not equal to $134$. Let $(M,g)$ be a closed, real-analytic Riemannian manifold with negative sectional curvature of dimension $n$. Suppose that $M$ contains a single complete, totally geodesic, immersed hypersurface that is not closed.
Then $M$ is isometric to a hyperbolic manifold.
\end{theoremintro}

\noindent At the moment we can obtain a similar result in all dimensions only by assuming in addition
that $M$ contains one closed totally geodesic immersed hypersurface or by assuming some curvature pinching to use the results of Cekic, Lefeuvre, Moroianu and Semmelmann \cite{CekicLefeuvreMoroianu2023_On-the-ergodicity-of-the-frame-flow-on-even-dimensional-manifolds}.  A general conjecture inspired by this result is also contained in \autoref{ssec:directions_for_future_work_and_conjectures}.

\subsubsection*{Proof Outline}
Our proof is based strongly on the fact that the set of hyperplanes in $TM$ tangent to totally geodesics submanifolds is a closed analytic set. This follows from the fact that the second fundamental form is an analytic function and totally geodesic submanifolds are characterized by vanishing of the second fundamental form, see \autoref{cor:analyticity_of_a_i_z}.  We can view this closed analytic set as living in $HM = \Gr_{n-1}(TM)$ the Grassman bundle of hyperplanes in $TM$. The goal is to show that this analytic set contains an open set, so is all of $HM$, at which point we are done by \autoref{thm:Cartan}.
A main input from the theory of analytic sets is that since this set contains infinitely many totally geodesic
manifolds of dimension $n-1$, it must be of dimension at least $n$.  In homogeneous dynamics one typically enlarges or regularizes an invariant set or measure by proving ``extra invariance'' under more group directions, here we use this jump in dimension from the theory of analytic sets
instead. While the argument is simple, it is also, to the best of our knowledge, novel and introduces a technique that should have other applications.

The rest of the proof involves exploiting this dimension jump. This is done in two steps, the first using the dynamics of the geodesic flow, the next using the dynamics of the frame flow. At the first step we use the dynamics of the geodesic flow to show that each unit tangent vector is contained in some totally geodesic hypersurface, see \autoref{prop:every_tangent_vector_is_contained_in_a_totally_geodesic}. This is done by showing that the set of such vectors that intersect either the unstable or stable manifold in an open set, which then implies they are dense under either the forward or backward flow.

In odd dimensions, the frame flow step follows from a theorem of Brin and Brin--Gromov.
If $UM$ denotes the unit tangent bundle of $M$ and $PM$ denotes the bundle of orthogonal frames over $M$, the main results of \cite{Brin1975_The-topology-of-group-extensions-of-C-systems, BrinGromov1980_On-the-ergodicity-of-frame-flows} imply that a closed invariant set in $PM$ covering all of $UM$ is all of $PM$.

In even dimensions the frame flow step is more novel and requires an analysis of closed invariant sets in $PM$ projecting to all of $UM$.  For this we use that the presence of even one closed immersed totally geodesic hypersurface $N$ in $M$ provides a great deal of transitivity of the frame flow on $PM$ using the Brin--Gromov result showing transitivity of the frame flow on $PN$. While this does not suffice to show transitivity on $PM$, it allows us to make a geometric argument reducing the transitivity on $HM$ to the fact that there are no non-vanishing vector fields on even dimensional spheres, see \autoref{prop:a_totally_geodesic_through_almost_every_hyperplane}.

We note that A.~Zeghib proved in \cite[Thm.~C]{Zeghib_Laminations-et-hypersurfaces-geodesiques-des-varietes1991} the analogue of \autoref{thm:finite} and \autoref{thm:unclosed} when the totally geodesic hypersurfaces are \emph{embedded}, i.e. without self-intersections.
The methods are completely different from ours and based on a lower bound of the volume of a neighborhood of an embedded hypersurface.  Zeghib's results hold in the same form for constant and variable curvature and so are very different in nature from ours which characterize constant curvature.

We also note that the theory of analytic sets has been used on occasion to study fixed point sets in
actions of large groups \cite{Bonatti, FarbShalen}. However, our use of it to study non-fixed sets invariant under flows seems novel.

Finally the techniques used in this paper are more or less entirely disjoint from those used in \cite{BFMS1,BFMS2}.
There is a common underlying idea of equidistribution but both how it is proven and how it is used are entirely different.

\subsubsection*{Acknowledgements} The authors thank Thang Nguyen for useful discussions concerning the
examples in \autoref{ssec:examples}, Yves Coud\`{e}ne for useful conversations about the criteria in \autoref{sssec:recurrence_in_the_noncompact_case}, Thibault  Lefeuvre for useful conversations concerning Brin groups, Michael Larsen for useful conversations concerning spin groups, and Laura DeMarco and Ralf Spatzier for enouraging us to consider effective finiteness and prove \autoref{prop:exactlyk}.
SF was supported by the NSF under grants DMS-2005470 and DMS-2305394. DF was supported by NSF under grant DMS-2246556 and also by Institut Henri Poincaré (UAR 839 CNRS-Sorbonne Université), and LabEx CARMIN (ANR-10-LABX-59-01). BL was supported NSF under grant DMS-2202830.



\section{Preliminaries from dynamics}
	\label{sec:preliminaries_from_dynamics}

\paragraph{Outline of section}
In \autoref{ssec:anosov_flows} we recall standard facts on Anosov flows.
Next, in \autoref{ssec:torsors_and_transitivity_groups} we take the opportunity to revisit some constructions of Brin \cite{Brin1975_The-topology-of-group-extensions-of-C-systems} regarding extensions of Anosov flows by compact groups.
We take the point of view described in \cite[\S2]{EskinFilipWright2018_The-algebraic-hull-of-the-Kontsevich-Zorich-cocycle} and \cite[\S5.2]{Filip_Translation-surfaces:-Dynamics-and-Hodge-theory}.
The picture is controlled by a family of closed subgroups of the compact group giving the extension.
A group is naturally associated to each point on the base of the flow, and for distinct points the groups are isomorphic, but not canonically.


\subsection{Anosov flows}
	\label{ssec:anosov_flows}

\subsubsection{Setup}
	\label{sssec:setup_anosov_flows}
Let $UM$ be a Riemannian manifold, for example the unit tangent bundle of another Riemannian manifold $M$.
A differentiable flow $g_t$ on $UM$ is is called \emph{Anosov} if there exists a continuous pointwise decomposition of the tangent bundle:
\[
	T_x UM = W^s(x) \oplus W^c(x) \oplus W^u(x)
\]
where $\dim W^c(x)=1$ and the vector field giving the flow spans $W^c$, and such that there exist $C,\lambda>0$ with
\begin{align*}
	\norm{g_t v} \leq C e^{-\lambda t}\norm{v}\quad &\forall v\in W^s(x),\forall t\geq 0\\
	\norm{g_{t} v} \leq C e^{\lambda t}\norm{v}\quad &\forall v\in W^u(x),\forall t\leq 0.
\end{align*}
When $M$ is a compact manifold with negative sectional curvature, the geodesic flow on its unit tangent bundle $UM$ is Anosov (\cite[\S17.6]{KatokHasselblatt1995_Introduction-to-the-modern-theory-of-dynamical-systems}).

We restrict to Anosov flows that preserve the volume induced by the Riemannian metric.

\subsubsection{Assumptions for stable/unstable manifolds}
	\label{sssec:assumptions_for_stable_unstable_manifolds}
Suppose that $(UM,g_t)$ is an Anosov flow on a not necessarily compact real-analytic manifold; when $M$ is compact, all that follows is trivially satisfied and the reader unfamiliar with these technical points can skip them.

All that we need for our arguments is the existence of stable and unstable foliations, as described in \autoref{sssec:stable_unstable_manifolds}, together with their standard properties as described below.
See also \cite{BarreiraValls2008_Analytic-invariant-manifolds-for-sequences-of-diffeomorphisms} for a recent treatment which shows that the invariant manifolds are real-analytic when the dynamics is.

We assume that there exists a continuous function $r\colon UM\to \bR_{>0}$ with the following properties:
\begin{description}
	\item[temperedness] There exists $C_1>1$ such that
	\[
		\tfrac 1 {C_1} \cdot r(x) \leq r(g_t x)\leq C_1 \cdot r(x)\quad \forall t\in[-1,1], \forall x\in UM.
	\]
	\item[tameness of dynamics] There exists $C_2>1$ such that denoting by $B(0,r(x))\subset T_x UM$ the ball of radius $r(x)$ in the tangent space, the maps
	\[
		g_t(x,r)\colon B(0,r(x))\xrightarrow{\exp_x} UM \xrightarrow{g_t}
		UM
		\xrightarrow{\exp_{g_t x}^{-1}}B(0,C_2\cdot r(g_tx))
	\]
	are well-defined for all $t\in [-1,1]$ and any $x\in M$.
	Furthermore, we assume that the maps $g_t(x,r)$ have holomorphic extensions $g_t(x,r)_{\bC}$ on the corresponding complexified balls and additionally satisfy the bounds $\norm{g_t(x,r)_{\bC}}_{C^0(B_\bC(0,r(x)/2))}\leq C_3$ for some constant $C_3>0$.
\end{description}
We use the norms from the complexification since, in general, placing a topology or a system of seminorms on real-analytic functions is rather subtle, see \cite[\S2.6]{krantz2002primer}.

\subsubsection{Stable/Unstable manifolds}
	\label{sssec:stable_unstable_manifolds}
With these assumptions, it follows from standard hyperbolic theory, as treated for example in \cite[\S7]{BarreiraPesin2007_Nonuniform-hyperbolicity} or \cite[\S4]{Pesin2004_Lectures-on-partial-hyperbolicity-and-stable-ergodicity}, that we have a foliation of $UM$ by immersed stable manifolds, with the leaf containing $x$ denoted $\cW^s[x]$ and the local manifold $\cW^s_{loc}[x]\subset \cW^s[x]\cap B(x,r(x))$ denoting the connected component containing $x$.

A defining property of stable manifolds is that $y\in \cW^s[x]$ if and only if $\dist(g_t y,g_tx)\xrightarrow{t\to +\infty}0$ (in fact, exponentially fast).
Note also that $T_x \cW^{s}[x]=W^{s}(x)$.

The analogous constructions work for the unstable manifolds $\cW^u[x]$ and reversed time direction.
We will also denote by $\cW^{cs/cu}[x]$ the center-stable (resp. center-unstable) manifolds obtained by applying the flow $g_t$ to the stable manifolds.

\subsubsection{Hopf argument}
	\label{sssec:hopf_argument}
The volume measure preserved by $g_t$ is ergodic, as follows from the Hopf argument, see e.g. \cite[Thm.~9.1.1]{BarreiraPesin2007_Nonuniform-hyperbolicity}.
In particular, by the Birkhoff ergodic theorem the set of points $D\subset UM$ for which both the forward and backward $g_t$-orbits are dense is a set of full measure.

Let $D^{+}\subset UM$ be the set of points for which the forward orbit under $g_t$ is dense in $UM$, and $D^-\subset UM$ defined analogously for the backward orbit.
It follows from the definition that $D^+$ is $\cW^s$-saturated (i.e. $x\in D^+$ and $y\in \cW^s[x]\implies y\in D^+$), and analogously for $D^-$ and $\cW^u$.
Since the sets are also $g_t$-invariant, it follows that they are in fact $\cW^{cs/cu}$-saturated sets of full measure.

In the case of negatively curved metrics, one can also verify density of $D^{+/-}$ from topological transitivity of the geodesic flow.  Transitivity of the flow can be proven geometrically, see e.g. \cite{Eberlein}.

\subsubsection{Center-stable projection}
	\label{sssec:center_stable_projection}
The foliations $\cW^{cs}$ and $\cW^{u}$ are everywhere transverse.
In particular, in any sufficiently small open set $W\ni x$ we have continuous maps
\begin{align*}
	\pi^{cs}_u \colon W &\to \cW^{u}_{loc}[x]\cap W\\
	y &\mapsto \cW^{cs}_{loc}[y]\cap \cW^{u}_{loc}[x]
\end{align*}
A submanifold $N\subset W$ with $\dim N = \dim \cW^u$ will be called \emph{transversal to the $\cW^{cs}$-foliation} if $T_x N\oplus W^{cs}(x)=T_x UM$.
In this case $\pi^{cs}_u$ is a homeomorphism between $N$ and its image (see also \cite[p.~236, Thm.~8.4.3]{BarreiraPesin2007_Nonuniform-hyperbolicity}).

We record the following useful consequence:
\begin{corollary}[Open set contains generic point]
	\label{cor:open_set_contains_forward_dense_point}
	For any $x\in UM$ and open subset $V\subset \cW^{u}_{loc}[x]$, the intersection $V\cap D^+$ is nonempty.
	The analogous statement for $\cW^{s}_{loc}$ and $D^-$ holds.
\end{corollary}

\begin{proof}
	For the open set $V\subset \cW^{u}_{loc}[x]$, let $W:=(\pi^{cs}_u)^{-1}V$ where the center-stable projection is relative to the point $x$.
	Then $W$ is open, hence has positive Lebesgue measure, hence contains elements of $D^-$.
	Since $D^-$ is $\cW^{cs}$-saturated, it follows that $V$ contains elements of $D^-$.
\end{proof}

\subsubsection{Recurrence and ergodicity in the noncompact case}
	\label{sssec:recurrence_in_the_noncompact_case}
Suppose $UM$ is the unit tangent bundle of a complete, finite volume Riemannian manifold $M$ which is not compact.
Suppose furthermore that all the sectional curvatures of $M$ lie in $\left[-C,\tfrac {-1} C\right]$ for some $C>0$.
A straightforward adaptation of \cite[Lemma~5.13]{FisherLafont_Finiteness-of-maximal-geodesic-submanifolds2021} implies that there exists a compact set $M_{th}\subset M$ such that every totally geodesic hypersurface intersects $M_{th}$.
Indeed, the necessary statement about virtual nilpotency of the fundamental groups of the cusps follows from the Margulis lemma, see e.g. \cite[\S8]{BallmannGromovSchroeder1985_Manifolds-of-nonpositive-curvature}.
The corresponding thick-thin decomposition is constructed in \cite[\S10.5]{BallmannGromovSchroeder1985_Manifolds-of-nonpositive-curvature}.

For our proof, we also need to know the geodesic flow on $UM$ is Anosov and ergodic. In finite volume this follows from requiring both that all the sectional curvatures of $M$ lie in $\left[-C,\tfrac {-1} C\right]$ for some $C>0$ and that all derivatives of the curvature tensor are uniformly bounded.  This is explained by Brin on page $82$ of his exposition of the proof in an appendix to \cite{BallmannLectures}.



\subsection{Torsors and Transitivity Groups}
	\label{ssec:torsors_and_transitivity_groups}

We collect some preliminary results on extensions of Anosov flows by compact groups.
The case of interest to us will be the frame flow, covering the geodesic flow on a unit tangent bundle.

\subsubsection{Setup}
	\label{sssec:setup_torsors_and_transitivity_groups}
Let $K$ be a compact Lie group.
All group actions and maps between manifolds are assumed to be real-analytic; all constructions work equally well within the smooth or continuous class of maps.

\begin{definition}[right $K$-torsor]
	\label{def:right_k_torsor}
	A \emph{right $K$-torsor} is a manifold $T$ equipped with a free and transitive action of $K$ on the right and denoted $(t,k)\mapsto t\cdot k$.

	A map of right $K$-torsors $f\colon T_1\to T_2$ is required to be equivariant for the right action of $K$, and we define:
	\[
		\Aut(T)^K:=\set{f\colon T\to T \text{ s.t. } f(t\cdot k)=f(t)\cdot k}.
	\]
\end{definition}
\begin{remark}[Properties of the groups and torsors]
	\label{rmk:properties_of_the_groups_and_torsors}
	\leavevmode
	\begin{enumerate}
		\item The actions on $T$ by $K$ on the right and by $\Aut(T)^{K}$ on the left commute, by definition.
		\item If we fix a basepoint $t_0\in T$ then we obtain an isomorphism between $T$ and $K$, compatible with the right $K$-action and that induces an identification of $\Aut(T)^K$ and $K$.
	\end{enumerate}
\end{remark}
Fix now a manifold $U$, with a flow $g_t\colon U\to U$.
\begin{definition}[Principal bundle]
	\label{def:principal_bundle}
	A $K$-principal bundle is a fibration $P\to U$ such that $P$ is equipped with a right action of $K$ that is free and transitive on the fibers.

	A $K$-cocycle on $P$ is a lift of the action of $g_t$ to $P$, denoted $G_t$, commuting with the right action of $K$.
\end{definition}

\subsubsection{Reduction of structure group}
	\label{sssec:reduction_of_structure_group}
Suppose $H\subset K$ is a subgroup and $P\to U$ is a $K$-principal bundle.
A \emph{reduction of the structure group to $H$} is a section $\sigma \colon U\to P/H$.
The regularity of the reduction (e.g. continuous, measurable, etc.) will refer to the regularity of $\sigma$.
A reduction of the structure group to $H$ is the same as specifying an $H$-principal subbundle $P_H\subset P$.

When the $K$-principal bundle admits a cocycle $G_t$, a reduction of the cocycle to $H$ means that the section $\sigma$ is $G_t$-invariant.

Note that given an $H$-principal bundle $P'\to U$ we can form the $K$-principal bundle
\[
	P:=P'\times_H K :=\rightquot{P'\times K}H \quad \text{with }(p',k')\cdot h := (p'h,h^{-1}k').
\]
The action of $K$ on $P$ is on the second factor, on the right: $(p',k')\cdot k:= (p',k'\cdot k)$ and it visibly commutes with the $H$-action when viewed on $P'\times K$.


\begin{example}[Framings]
	\label{eg:framings}
	Suppose $V$ is an $n$-dimensional vector space with a nondegenerate positive definite quadratic form and $\bR^n$ is equipped with its standard Euclidean quadratic form.
	Then
	\[
		T:=\Isom(\bR^n,V)=\set{\phi\colon \bR^n\to V \text{ isometrically}}
	\]
	is naturally a right $\Orthog(\bR^n)$-torsor with action by pre-composition, such that $\Aut(T)^K=\Orthog(V)$ with action on the left by post-composition.

	Suppose now that $M$ is an $n$-dimensional Riemannian manifold.
	Then applying the previous construction to each tangent space we obtain the $\Orthog(\bR^n)$-principal bundle $PM\to M$ of frames. 
\end{example}

\subsubsection{Stable manifolds on the principal bundle}
	\label{sssec:stable_manifolds_on_the_principal_bundle}
Suppose next that the flow $g_t$ on $U$ is Anosov, and $P\to U$ is a cocycle.
Then the flow direction and stable/unstable subbundles have lifts $\wtilde{W}^{s/u/c}(p)$ and we have a $G_t$-invariant decomposition
\[
	T_{p}P = T^vP(p) \oplus \wtilde{W}^{s}(p) \oplus \wtilde{W}^u(p)
	\oplus \wtilde{W}^c(p)
\]
where $T^vP\subset TP$ denotes the vertical subbundle (the kernel under the projection $P\to U$), and $\wtilde{W}^c(p)$ is the direction of the vector field generating $G_t$.
Furthermore, the flow $G_t$ on $P$ also admits stable/unstable manifolds $\wtilde{\cW}^{s/u}[p]$ with the same defining properties as in \autoref{sssec:stable_unstable_manifolds}:
$p'\in \wtilde{\cW}^s[p]$ if and only if 
$\dist(G_t p',G_t p)\xrightarrow{t\to +\infty} 0$ and
$T_p \wtilde{\cW}^s[p]=\wtilde{W}^s(p)$.
If $p\in P$ projects to $u\in U$, then $\wtilde{\cW}^s_{loc}[p]$ projects bijectively to $\cW^s_{loc}[u]$.
For the existence and properties of stable/unstable foliations in this context, we refer to \cite[\S4]{Pesin2004_Lectures-on-partial-hyperbolicity-and-stable-ergodicity}.

\subsubsection{Holonomy on Principal Bundles}
	\label{sssec:holonomy_on_principal_bundles}
Denote the fiber of $P$ above $u\in U$ by $P(u)$.
For $u'\in \cW^s_{loc}[u]$ define the \emph{holonomy} map:
\begin{align*}
	\cH^{s}(u,u') \colon P(u)&\to P(u')\\
	p& \mapsto \wtilde{\cW}^s_{loc}[p]\cap P(u')
\end{align*}
Equivalently, $p$ gets mapped to the unique point $p'\in P(u')$ for which $\dist(G_t p, G_t p')\to 0$ as $t\to +\infty$.

The holonomy map is continuous in $u,u'$, since the stable foliation is continuous.
Furthermore for fixed $u,u'$ it is a map of $K$-torsors, since for a fixed $k\in K$ we have
\[
	\dist\left(G_t p, G_t p'\right)\xrightarrow{t\to +\infty} 0	\iff
	\dist\left(G_t p\cdot k, G_t p'\cdot k\right)\xrightarrow{t\to +\infty} 0.
\]
Note also that the holonomy map can be defined for any $u'\in \cW^s[u]$, by applying $g_{t_0}$ for a fixed but sufficiently large $t_0$ such that $g_{t_0} u'\in \cW^s_{loc}[g_{t_0}u]$.

The holonomy map along unstable manifolds $\cH^u$ is defined analogously.
We can similarly define the center-stable and center-unstable holonomies $\cH^{cs},\cH^{cu}$, with the same properties.

\subsubsection{Generic part of a closed set}
	\label{sssec:generic_part_of_a_closed_set}
A point $u_0\in U$ will be called \emph{forward or backward generic} if either its forward or its backward orbit under $g_t$ is dense in $U$.
Recall from \autoref{sssec:hopf_argument} that the sets of these points are denoted $D^{\pm}$.
A point that is both forward and backward generic will be called \emph{two-sided generic}, and a \emph{generic} point will be one that is either forward or backward generic (or possibly both).
For a set $X\subset P$, we will denote by $X(u):=X\cap P(u)$.

Now if $X\subset P$ is a closed set, and $u_0\in U$ is generic, we denote by
\[
	X_{gen,u_0}:= \ov{\bigcup_{t\in \bR} G_t \cdot X(u_0)}
\]
the orbit closure of a generic fiber. 
Note that $X_{gen,u_0}$ is empty if $X(u_0)$ is empty.

\begin{proposition}[Properties of the generic part]
	\label{prop:properties_of_the_generic_part}
	Suppose $X\subset P$ is a closed $G_t$-invariant set, and fix a generic $u_0 \in U$.  
	\leavevmode
	\begin{enumerate}
		\item For any other generic $u_1\in U$ we have that
		\[
			X_{gen,u_0} = X_{gen, u_1}.
		\]
		In particular, the set is independent of the choice of $u_0$, which we omit from notation, and we additionally have:
		\[
			X_{gen} = 
			\ov{\bigcup_{u\in D^+} X(u)}
			=
			\ov{\bigcup_{u\in D^-} X(u)}
			=
			\ov{\bigcup_{u\in D^+\cap D^-} X(u)}
		\]
		i.e. $X_{gen}$ is the closure of all the generic fibers.  We will refer to $X_{gen}$ as the \textit{generic part} of $X$.  
		\item The set $X_{gen}$ is holonomy-invariant, namely for any $u\in U$ and any $u'\in \cW^{s}[u]$ we have that
		\[
			\cH^{s}\left(u,u'\right)X_{gen}(u) = X_{gen}\left(u'\right)
		\]
		and similarly for $\cH^{u}$.
	\end{enumerate}
\end{proposition}
\begin{proof}
	For part (i), let $u_0,u_1\in U$ be generic.
	For $x\in X(u)$ define:
	\[
		KX_{x}:=\set{k\in K\colon x\cdot k \in X(u)}
	\]
	Note that this is a closed subset of $K$, since the fibers $X(u)$ are closed.
	Fix $x_0\in X(u_0)$.

	Suppose that $g_{t_i}u_0\to u_1$ along a sequence of times $|t_i|\to +\infty$, and passing to a subsequence assume $G_{t_i}x_0\to x_1\in X(u_1)$.
	Then, since the action of $K$ commutes with the flow, we have that
	\[
		G_{t_i}\left(x_0 \cdot k\right) \to x_1\cdot k \text{ so }
		KX_{x_0}\subseteq KX_{x_1}.
	\]
	Applying the same argument with a sequence of times $|t_i'|\to +\infty$ such that $g_{t_i'}u_1\to u_0$, and $G_{t_i'}x_1\to x_0'\in X(u_0)$, we find that $KX_{x_1}\subseteq KX_{x_0'}$.
	If $x_0'=x_0\cdot k$ then $KX_{x_0'}=k^{-1}\cdot KX_{x_0}$ so we conclude that
	\[
		KX_{x_0}\subseteq KX_{x_1}\subseteq k^{-1}\cdot KX_{x_0}.
	\]
	To conclude, it suffices to show that the inclusions are in fact equalities.
	Suppose by contradiction that one of them is not, so $KX_{x_0}\subsetneq k^{-1} KX_{x_0}$, i.e. $k\cdot KX_{x_0}\subsetneq KX_{x_0}$.
	Since these are closed sets, let $k'\in KX_{x_0}$ be such that $\dist(k',k\cdot KX_{x_0})>0$.
	Note that we obtain a sequence of nested sets $k^{n+1}KX_{x_0}\subsetneq k^{n} KX_{x_0}$ for $n\geq 0$.

	Now, since the group $K$ is compact, the sequence $\set{k^{n}}_{n\in \bZ}$ accumulates at the identity element $\id$, so there exists a subsequence with $k^{n_i}\to \id$ with $n_i>0$ (note that if $k^{n_i}\to \id$ then $k^{-n_i}\to \id$ as well).
	Therefore $k^{n_i}\cdot k'\to k'$ which is a contradiction.

	For (ii), holonomy invariance, we prove the statement first under the assumption that $u$ is two-sided generic.
	Since $u'\in \cW^s[u]$, then $u'$ is also forward-generic (since its forward orbit is asymptotic to that of $u$).
	Fix a forward-generic $u''$ and a sequence of times $t_i\to +\infty$ such that $g_{t_i}u\to u''$ (and thus $g_{t_i}u'\to u''$).
	Note that $X$ and $X_{gen}$ have the same fibers over $u,u',u''$.

	Suppose by contradiction that $x'=\cH^s(u,u')x\notin X(u')$.
	Passing to a subsequence, suppose $G_{t_i}x\to x''\in X(u'')$, and since $x'\in \cW^s[x]$ we also have $G_{t_i}x'\to x''$.
	By part (i), there also exists $x_1'\in X(u')$ such that, upon passing to a further subsequence, we also have $G_{t_i}x_1'\to x''$.
	This is a contradiction, since $d_{P(u')}(x',X(u'))>0$ and the cocycle action is by isometries.
	Note that the same argument applies with the roles of $u,u'$ reversed, so the proof is complete when one point is two-sided generic.

	It now suffices to prove (ii) when $u$ is arbitrary, while $u'\in \cW^s_{loc}[u]$ is backward-generic, since $\cH^s(u,u'')=\cH^s(u',u'')\circ \cH^s(u,u')$ and because the set of backward-generic points on $\cW^s[u]$ is nonempty (and in fact dense, see \autoref{sssec:hopf_argument}).
	
	Fix $x\in X(u)$ and let $x':=\cH^s(u,u')x$.
 We will construct a sequence $x_i'\in X_{gen}$ with $x_i'\to x' $.
	Let $u_i\to u$ be a sequence of two-sided generic points converging to $u$.
	Using the local product structure, let now $u_i':=\cW^s_{loc}[u_i]\cap \cW^{cu}_{loc}[u']$ be a sequence of forward-generic points (since they are on the same stable leaf as $u_i$).
 	Note that in fact $u_i'$ are two-sided generic since they are on the same center-unstable leaf as $u' $, and since $u'$ is backward-generic it follows that $u_i'$ is backward-generic.
 
	Then $u_i'\to u'$ since $u_i\to u$.  
	We work in an open neighborhood $\cU$ of $u$ where $P\vert_{\cU}\isom \cU\times K$ as a smooth principal bundle.
	Then, all the holonomy maps are continuous in this trivialization, and depend continuously on their parameters.
	So we have that $\cH^{s}(u_i,u_i')\to \cH^s(u,u')$ as maps of $K$.
	We also know that $\cH^s(u_i,u_i')X(u_i)=X(u_i')$ by the earlier half of the argument.

	By the definition of $X_{gen}$, we know that any $x\in X_{gen}(u)$ is accumulated by some sequence $x_i\in X(u_i)$.
	So
	\[
		\cH^s(u,u')x = \lim \cH^s(u_i,u_i')x_i \in X(u').
	\]
	Conversely, take any $x'\in X(u')$, then $\cH^s(u',u)x'$ is the accumulation point of the sequence
	\[
		\cH^s(u_i',u_i) \cH^{cu}(u',u_i')x'\in X(u_i)
	\]
	since $\cH^{cu}(u',u_i')\to \id$ in the above local trivialization, and the $u_i'$ are two-sided generic.  
	This accumulation point belongs to $X_{gen}(u)$ and this concludes the proof.  
\end{proof}

\subsubsection{$su$-chains}
	\label{sssec:su_chains}
A sequence of points $x_1,\dots, x_n$ will be called an \emph{$su$-chain} if any two consecutive points belong either to the same center-stable manifold, or to the same center-unstable manifold.
Associated to an $su$-chain we obtain a map
\[
	\cH(x_1,\dots,x_n)\colon P(x_1)\to P(x_n)
\]
obtained by composing the holonomy maps between consecutive points.
Note that the map depends on the choice of $su$-chain, not just endpoints.

Denote by $K(x):=\Aut\left(P(x)\right)^K$ the automorphism group of the right $K$-torsor $P(x)$, for $x\in M$.
For an $su$-chain that starts and ends at $x$, the map constructed above belongs to $K(x)$.

\begin{definition}[Brin group]
	\label{def:brin_group}
	For each $x\in U$, define $B(x)\subseteq K(x)$ to be the closure of the group generated by all the maps constructed above, ranging over all $su$-chains that start and end at $x$.
\end{definition}
Note that for any $x,x'\in U$ and a choice of $su$-chain connecting them, the holonomy map gives an identification of the two groups $B(x),B(x')$.

\begin{corollary}[Invariance under transitivity group]
	\label{cor:invariance_under_transitivity_group}
	Suppose that $X\subset P$ is closed and $G_t$-invariant.
	Let $X_{gen}\subset X$ be its generic part, as in \autoref{sssec:generic_part_of_a_closed_set}.

	Then $X_{gen}(u)$ is invariant under $B(u)$.
\end{corollary}
\begin{proof}
	This follows immediate from \autoref{prop:properties_of_the_generic_part} and the definition of $B(u)$.
\end{proof}

\subsubsection{Minimal sets}
	\label{sssec:minimal_sets}
Suppose that $X\subset P$ is closed, $G_t$-invariant, and projects surjectively to $U$.
We will say that $X$ is \emph{relatively minimal} if it is minimal under inclusion among all sets with this property.
Since fibers of $P\to U$ are compact, relatively minimal sets exist as we can intersect a descending chain.

Note that if a set $X$ is relatively minimal, then it coincides with its generic part $X_{gen}$, since $X_{gen}\subseteq X$ is closed and projects surjectively to $U$ if $X$ does.
Additionally, if $x\in X$ projects to $u\in U$ such that $u$ is generic, then $x$ has dense orbit in $X$.

\begin{proposition}[Minimal sets are fiberwise a single $B$-orbit]
	\label{prop:minimal_sets_are_fiberwise_a_single_b_orbit}
	Suppose $X\subset P$ is a relatively minimal set.
	Then for every $u\in U$ the set $X(u)$ is a single $B(u)$-orbit.
\end{proposition}
\begin{proof}
	Since $X=X_{gen}$ is invariant under holonomy by \autoref{prop:properties_of_the_generic_part}, it suffices to prove the statement for $u\in U$ in a dense set.
	Pick a $u\in U$ with dense $g_t$-orbit and suppose $X(u)$ contains two distinct $B(u)$-orbits $O_1,O_2\subset X(u)$.
	Fix $\ve>0$ such that $\dist(O_1,O_2)>\ve$ where the distance between the two sets is in the Hausdorff sense.
	It follows in particular that $\forall p_i\in O_i, \forall b\in B(u)$ we have that $\dist(p_2, bp_1)>\ve$.

	Note also that since the holonomy maps are continuous, for any $\ve_1>0$ there exists $\delta_1>0$ such that if $\dist(u,u')<\delta_1$ then there exists an $su$-chain connecting $u,u'$ such that the associated holonomy map $\cH(\bullet)\colon P(u')\to P(u)$ satisfies
	\[
		\dist(p',\cH(\bullet)p')<\ve_1 \quad \forall p'\in P(u').
	\]
	Pick now $p_1\in O_1$, since $X$ is minimal it follows that the orbit of $p_1$ is dense in $X$.
	Consider a sequence of times $t_i$ such that $G_{t_i}p_1\to p_2\in O_2$.
	Clearly we must also have $g_{t_i}u\to u$.
	For $\ve_1:=\ve/10$, there exists $\delta_1>0$ such that, for $i$ sufficiently large, we have that $\dist(G_{t_i}p_1,p_2)<\ve/10$ and there exists an $su$-chain between $g_{t_i}u$ and $u$, with associated holonomy map $\cH(\bullet)$ not moving points by more than $\ve/10$.
	Now $b:=\cH(\bullet)\circ G_{t_i}\in B(u)$ and we have $\dist(p_2,b p_1)<\ve/5$ which is a contradiction.
\end{proof}

\begin{proposition}[Reduction of structure group]
	\label{prop:reduction_of_structure_group}
	Suppose $X\subset P$ is a relatively minimal set.
	\begin{enumerate}
		\item There exists a closed subgroup $H\subset K$ such that $H$ acts freely and transitively (on the right) on the fibers $X(u)$, for any $u\in U$.
		\item For any $u\in U$, after choosing a basepoint $p_0\in P$ the group $B(u)$ is identified with a conjugate of $H$.
	\end{enumerate}
\end{proposition}
\begin{proof}
	Since $X$ is relatively minimal, for any $k\in K$ we have that $X\cdot k \cap X$ is either empty or all of $X$.
	Define now
	\[
		H:=\set{k\in K \colon X\cdot k = X}.
	\]
	By construction $H$ acts transitively on any fiber $X(u)$.
	By \autoref{prop:minimal_sets_are_fiberwise_a_single_b_orbit} the group $B(u)$ also acts transitively on $X(u)$.

	Choosing a basepoint $p_0\in P(u)$ identifies $X(u)\toisom x_0 H$ with a right $H$-coset, and the group that acts transitively on it on the left is $x_0 H x_0^{-1}$, which gets identified with $B(u)$.
\end{proof}
\begin{corollary}[Reduction from minimal set]
	\label{cor:reduction_from_minimal_set}
	Any relatively minimal set $X\subset P$ gives a closed subgroup $H\subset K$ and a continuous reduction of the structure group of the cocycle on $P$ to $H$.
\end{corollary}
\begin{proof}
	The relatively minimal set $X$ gives a map $[X]\colon U \to P/H$, which is continuous because $X$ is closed and the fibers of the principal bundle are compact.
	This map, or the corresponding $H$-principal subset $X\subset P$, is the desired reduction.
\end{proof}

\begin{proposition}[Closed invariant sets]
	\label{prop:closed_invariant_sets}
	Let $H\subset K$ be a closed subgroup of $K$ arising from a relatively minimal set as in \autoref{cor:reduction_from_minimal_set} and suppose that $K_1\subset K$ is a closed subgroup.
	Denote by $G_t$ the flow on the quotient space $P/K_1$.

	There exists a continuous map
	\[
		cl\colon \rightquot{P}{K_1} \to \leftrightquot{H}{K}{K_1}
	\]
	which is constant on $G_t$-orbits on the left.
	The sets that arise as the generic part of closed $G_t$-invariant subsets of the quotient $P/K_1$ are in bijection with closed sets on the right, by taking the inverse image.
\end{proposition}
\begin{proof}
	It suffices to treat the case $K_1=\set{\id}$ and ensure that the map $cl$ is $K$-equivariant.
	As in the proof of \autoref{prop:reduction_of_structure_group}, let $X\subset P$ be a relatively minimal set leading to the group $H\subset K$.
	Then the disjoint $K$-translates of the minimal set $X$ are equivariantly identified with $\leftquot{H}{K}$, and this gives the map $cl\colon P\to \leftquot{H}{K}$.
	The map $cl$ is continuous since the section $[X]\colon U\to P/H$ is continuous.
\end{proof}

\begin{definition}[Transitivity group]
	\label{def:transitivity_group}
	A group $H\subset K$ arising from a relatively minimal set as in \autoref{cor:reduction_from_minimal_set} will be called a \emph{transitivity group}.
\end{definition}

\begin{remark}[On transitivity groups]
	\label{rmk:on_transitivity_groups}
	\leavevmode
	\begin{enumerate}
		\item The Brin groups (\autoref{def:brin_group}) and the transitivity group are isomorphic, but are identified only up to conjugation by elements of $K$.
		\item After passing to a finite cover of the base manifold, we can ensure that the transitivity group is connected, see \cite[Lemma~3.3]{CekicLefeuvreMoroianu2023_On-the-ergodicity-of-the-frame-flow-on-even-dimensional-manifolds}.
		We do so in the applications in \autoref{sec:proof_of_the_main_theorem}.
	\end{enumerate}
\end{remark}




\section{Proof of the Main Theorem}
	\label{sec:proof_of_the_main_theorem}


\subsection{Proof outline}
	\label{ssec:proof_outline}

\subsubsection{Diagram of vectors and subspaces}
	\label{sssec:diagram_of_vectors_and_subspaces}
Let $UM\xrightarrow{\pi_1} M$ denote the unit sphere bundle, $HM\xrightarrow{\pi_{n-1}}M$ the Grassmannian bundle of oriented hyperplanes, i.e. $(n-1)$-dimensional tangent planes, and $FM$ the bundle of pairs $(v,\Pi)$ where $v$ is unit vector, $\Pi$ is an oriented hyperplane, and $v\in \Pi$.
Denote the forgetful projections by $FM\xrightarrow{\pi_U}UM$ and $FM\xrightarrow{\pi_H}HM$.

Denote next by $Z\subset HM$ the subset of hyperplanes that are tangent to a germ of totally geodesic hypersurface (with either choice of orientation).
Set now $I:=\pi_{H}^{-1}(Z)\subset FM$ and $A:=\pi_U(I)$, so that $A$ denotes the set of unit tangent vectors through which there exists a totally geodesic hypersurface.

\begin{equation}
	\label{eqn_cd:SM_FM_HM_M}
	\begin{tikzcd}
		&
		FM
		\arrow[dl, "\pi_U", swap]
		\arrow[dr, "\pi_H"]
		&
		\\
		UM
		\arrow[dr, "\pi_1", swap]
		&
		&
		HM
		\arrow[dl, "\pi_{n-1}"]
		\\
		&
		M
		&
	\end{tikzcd}
	\qquad
	\begin{tikzcd}
		&
		I
		\arrow[dl, "\pi_U", swap]
		\arrow[dr, "\pi_H"]
		&
		\\
		A
		&
		&
		Z
		\\
		&
		&
	\end{tikzcd}
\end{equation}
Our notation is summarized in \autoref{eqn_cd:SM_FM_HM_M}.
Typically, a point in $M$ will be denoted $p$, a point in $HM$ will be denoted $(p,\Pi)$ with $\Pi\subset T_pM$, and a point in $UM$ will be denoted $(p,v)$ with $v\in T_p M$.

\subsubsection{Main steps in the proof}
	\label{sssec:main_steps_in_the_proof}
We can now describe our proof with more precision. It is divided into the following steps:
\begin{description}
	\item[Step 1] The set $A$ coincides with $UM$.
	In other words, through every tangent vector there exists at least one totally geodesic hypersurface.
	This is established in \autoref{ssec:geometry_of_real_analytic_sets}, using the geometry of real-analytic sets and the ergodic theory of the geodesic flow.
	\item[Step 2] The set $Z$ is all of $HM$.
	This is established in \autoref{ssec:analyzing_the_transitivity_group}, using the constraints coming from the transitivity group.
	\item[Step 3] We conclude by referring to Cartan's \autoref{thm:Cartan}, see \autoref{sssec:real_hyperbolic_last_step}.
\end{description}



\subsection{Geometry of real-analytic sets}
	\label{ssec:geometry_of_real_analytic_sets}

\subsubsection{Setup}
	\label{sssec:setup_geometry_of_real_analytic_sets}
We keep the notation of \autoref{sssec:main_steps_in_the_proof}.
It is immediate from the definitions that $Z\subset HM$ is a closed subset, since a totally geodesic hypersurface $S\ni p$ with $T_pS = \Pi$ must locally be equal to the exponential map of a ball in $\Pi$, and a hypersurface is totally geodesic if and only if its second fundamental form vanishes identically (see \cite[\S12-13]{ONeill1983_Semi-Riemannian-geometry}).

We will call a set \emph{real-analytic} if it is locally the vanishing locus of finitely many real-analytic functions.
A subset is \emph{subanalytic} if it is locally the image, under a real-analytic map, of a relatively compact real-analytic subset (see \cite[Def.~3.1]{BierstoneMilman1988_Semianalytic-and-subanalytic-sets}).
Many equivalent descriptions of subanalytic sets exist, see \cite[Prop.~3.13]{BierstoneMilman1988_Semianalytic-and-subanalytic-sets}.

As a preliminary, we need:
\begin{lemma}[Fiberwise vanishing and real-analyticity]
	\label{lem:fiberwise_vanishing_and_real_analyticity}
	Suppose $\pi\colon X\to Y$ is a real-analytic fibration of real-analytic manifolds, such that the fibers, denoted $X_y$, are connected real-analytic manifolds.
	Let $\rho\colon X\to \bR$ be a real-analytic function and let $Z\subset Y$ be the set of points on the base such that $\rho$ vanishes identically on the fibers, namely:
	\[
		Z = \set{y\in Y\colon \rho\vert_{X_y}\equiv 0}.
	\]
	Then $Z$ is a real-analytic set.
\end{lemma}
\begin{proof}
	The statement is local (by connectedness of the fibers) so by an appropriate choice of coordinates on $X$ we can assume that $\pi$ is a coordinate projection in a ball of $\bR^{m+n}$ to $\bR^m$.
	Using a multi-index notation with $(\bbx,\bby)\in \bR^m\times \bR^n$, we can expand in Taylor series: $\rho(\bbx,\bby)=\sum_{I}a_I(\bbx)\bby^{I}$ and then $Z$ is given as the vanishing locus of all the $a_I(\bbx)$.
	However, by \cite[Ch.~V, Cor.~2, pg.~100]{Narasimhan1966_Introduction-to-the-theory-of-analytic-spaces} it follows that $Z$ can be given as the vanishing locus of finitely many real-analytic functions in $\bbx$.
\end{proof}

\begin{corollary}[Analyticity of $A, I, Z$]
	\label{cor:analyticity_of_a_i_z}
	With definitions as in \autoref{sssec:diagram_of_vectors_and_subspaces}:
	\begin{enumerate}
		\item The sets $I\subset FM$ and $Z\subset HM$ are real-analytic.
		\item The set $A\subset UM$ is subanalytic.
	\end{enumerate}
\end{corollary}
\begin{proof}
	It suffices to prove that $Z$ is real-analytic, since the corresponding statements for $I,A$ follow from the definitions of the sets.
	
	Let $UHM\to HM$ be the tautological vector bundle whose fiber over the hyperplane $\Pi\in HM$ is the vector space $\Pi$.
	Define the following adaptation of the exponential map
	\begin{align*}
		E\colon UHM& \to HM\\
		\big(p,\Pi,v\big) & \mapsto \big(\exp_p(v), D_v \exp_p(\Pi)\big)
	\end{align*}
	where $\exp_p\colon T_pM\to M$ denotes the exponential map, $p\in M,\Pi\subset T_pM$ and $v\in \Pi$.
	In other words, the map $E$ takes $(p,\Pi,v)$ to a corresponding hyperplane at $\exp_p(v)$.

	Define now $\rho\colon HM\to \bR$ by $\rho(p,\Pi):=\norm{II_{\exp_p(\Pi)}(p)}^2$ where $\exp_p(\Pi)$ denotes the hypersurface obtained by exponentiating $\Pi$, $II$ is its second fundamental form, and thus $\rho(p,\Pi)$ is its norm squared evaluated at $p$.
	By construction $\rho$ and the pullback of $\rho$ along $E$, denoted $U\rho:=E^*\rho \colon UHM\to \bR$, are real-analytic.

	Recall now that a hypersurface is totally geodesic if and only if its second fundamental form vanishes identically.
	Therefore $(p,\Pi)\in Z$ if and only if $E^*\rho(p,\Pi,\bullet)$ vanishes identically on the fiber of $UHM$ over $(p,\Pi)$.
	By \autoref{lem:fiberwise_vanishing_and_real_analyticity} it follows that $Z$ is real-analytic.
\end{proof}

\subsubsection{Structure of real-analytic sets}
	\label{sssec:structure_of_real_analytic_sets}
By \cite[2.11-12]{BierstoneMilman1988_Semianalytic-and-subanalytic-sets} an analytic set $Z$ has a stratification into finitely many locally closed pieces that are analytic submanifolds.
The maximal dimension of a stratum is called the dimension of $Z$.

Suppose now that $M$ has infinitely many closed, immersed, distinct totally geodesic hypersurfaces $\{S_i\}_{i\in \bN}$ as in \autoref{thm:main}.
We have the double covers $\wtilde{S}_i\to S_i$ with $\wtilde{S}_i\subset HM$ given by the two choices of orientation of a tangent plane.
Note that $\wtilde{S}_i\subset HM$ is a closed embedded submanifold.
Now we also have that $\wtilde{S}_i\subset Z$ are disjoint, since a tangent hyperplane determines a unique totally geodesic hypersurface (if it exists).

Since $\dim \wtilde{S}_i=\dim S_i=n-1$ and infinitely many are contained in $Z$, it follows that $\dim Z\geq n$.
Now $\dim I = \dim Z + (n-2)$ and so $\dim I\geq 2n-2$.
Our first goal is to prove:
\begin{proposition}[Every tangent vector is contained in a totally geodesic]
	\label{prop:every_tangent_vector_is_contained_in_a_totally_geodesic}
	We have that $A=UM$, i.e. through every unit tangent vector there exists at least one totally geodesic hypersurface.
\end{proposition}

\begin{proof}
	Let $A^{\circ}\subset A$ denote the set of smooth points, i.e. it is real-analytic, with $\dim A^{\circ} = \dim A$, and $\dim\left(A\setminus A^{\circ}\right)<\dim A$ see e.g. \cite[Theorem 7.2]{BierstoneMilman1988_Semianalytic-and-subanalytic-sets}.
	Consider the map $\pi_U\colon I \to A$.
	Either $\dim A=\dim I$ (call this Case 1), or there exists a further relatively open subset $A^{\circ\circ}\subset A^{\circ}$ such that the map $\pi_{S}\vert_I$ has fibers of positive-dimension over $A^{\circ\circ}$ (call this Case 2).

	In either case, we will prove that there exists $a\in A$ that has a forward (or backward) $g_t$-orbit which is dense in $UM$, and since $A$ is closed and $g_t$-invariant it follows that $A=UM$.

	\paragraph{Case 1}
	Let $a\in A^{\circ}$ be a smooth point and recall that $\dim A=\dim I\geq 2n-2$.
	Recall that $\dim UM=2n-1$, so if the inequality were strict we would be done.
	Therefore we can assume that $\dim A = 2n-2$, i.e. that $A$ is a hypersurface in $UM$.

	We have that $T_a UM=W^{s}(a)\oplus W^c(a)\oplus W^{u}(a)$ and $W^c(a)\subset T_a A$, since it represents the geodesic flow direction.
	In the quotient $T_a UM/W^c(a)$ the hyperplane $T_a A/W^c(a)$ must be transverse to at least one of the images of $W^{u/s}(a)$, say $W^s(a)$. Therefore the intersection $W^{cs}(a)\cap T_a A$ must have the expected dimension, so we have the equalities
	\[
		\dim \left(W^{cs}(a)\cap T_a A\right)=\dim (W^{cs}(a))-1 = n-1.
	\]
	Pick an $(n-1)$-dimensional germ of smooth real-analytic submanifold $N\subset A^{\circ}$ with $a\in N$ such that $T_a UM = T_a N\oplus W^{cs}(a)$.
	For example, $N$ can be obtained by intersecting $A^{\circ}$ with a linear space of the appropriate dimension in local coordinates.
	We thus have that in a neighborhood of $a$, the manifold $N$ is transverse to the center-stable foliation.

	It follows from the properties of the center-stable projection as defined in \autoref{sssec:center_stable_projection} that there exists a neighborhood $U\subset UM$ of $a$ such that $\pi^{cs}_u(N\cap U)\subset \cW^u_{loc}[a]$ contains a nonempty open set in $\cW^{u}_{loc}[a]$.
	From \autoref{cor:open_set_contains_forward_dense_point} there exists $s=\pi^{cs}_u(n)$, with $n\in N$ which is forward-dense, and since the set of forward-dense points is center-stable saturated it follows that $n\in N \subset A$ is forward-dense.

	\paragraph{Case 2}
	Suppose that $a=(p,v)\in A^{\circ\circ}$ is such that $I_a:=\pi_1^{-1}(a)\subset I$ contains a real-analytic curve; to ease notation, we will continue to denote that real-analytic curve by $I_a$.
	We will work in a sufficiently small neighborhood $U$ of $p\in M$.

	For $i:=(p,v,\Pi)\in I_a$ denote by $H_i\subset U$ the totally geodesic hypersurface corresponding to the hyperplane $\Pi\subset T_pM$.
	We claim that the union $\cup_{i\in I_a} H_i$ contains an open subset $U_1$ of $U$.
	Indeed, this is true at the level of tangent spaces: the image of a $1$-parameter family of linear hyperplanes contains an open set.
	The exponential map is a local diffeomorphism and the claim follows.

	The map $\pi_1\colon \cW^{cs}_{loc}[a]\to M$ is a diffeomorphism onto its image (see for example \cite[\S12.2]{BarreiraPesin2007_Nonuniform-hyperbolicity}).
	Therefore, there exists an inverse map $\sigma_{a}\colon U \to \cW^{cs}_{loc}[a]$ which is again a diffeomorphism onto its image.
	For $p\in U$, the unit tangent vector $\sigma_a(p)$ is characterized by the property that it is the unique vector above $p$ such that $g_t(\sigma_a(p))$ and $g_t a$ remain at bounded distance as $t\to +\infty$.

	Suppose now that $p_1\in U_1$ and $p_1\in H_{i_1}$ for some $i_1\in I_a$.
	Recall that $H_{i_1}$ is a totally geodesic hypersurface that passes through $p_1$ and $p$.
	Since $H_{i_1}$ is totally geodesic, there exists a unique vector $v_1$ above $p$, tangent to $H_{i_1}$, such that $g_t v_1$ and $g_t a$ are at bounded distance as $t\to +\infty$.
	Therefore, $v_1=\sigma_a(p_1)$ and thus $\sigma_{a}(p_1)$ is tangent to $H_{i_1}$.
	We conclude that $\sigma_a(p_1)\in A$ by the construction of $A$.

	It now follows that $A\cap \cW^{cs}_{loc}[a]$ contains the open set $\sigma_a(U_1)$.
	Again, we conclude from \autoref{cor:open_set_contains_forward_dense_point} that $A$ contains a backward-dense point and hence $A=UM$.
\end{proof}





\subsection{Analyzing the transitivity group}
	\label{ssec:analyzing_the_transitivity_group}

As we saw in \autoref{cor:reduction_from_minimal_set}, there exists a continuous reduction of the structure group of the frame bundle to the transitivity group.
Restricting this reduction to the unit sphere over one point of $M$ already yields nontrivial obstructions, first exploited by Brin and Gromov \cite{BrinGromov1980_On-the-ergodicity-of-frame-flows}.
A systematic analysis of known obstructions is given in \cite[\S3.3]{CekicLefeuvreMoroianu2023_On-the-ergodicity-of-the-frame-flow-on-even-dimensional-manifolds}, which we use to extract the needed constraints on the transitivity group.
While ergodicity of the frame flow is not known in the generality we would need, the existing restrictions coupled with additional arguments are sufficient for our purposes.

\subsubsection{The simplest obstruction}
	\label{sssec:the_simplest_obstruction}
Recall that even-dimensional spheres have Euler characteristic $2$, hence do not admit a nowhere vanishing vector field, or even an everywhere defined tangent line-field. This is originally due to Poincar\'{e} in dimension $2$ \cite[Chapter 13]{Poincare} and Brouwer in higher dimensions \cite{Brouwer}.
In other words, if $n\geq 3$ is odd, there does not exist a continuous map
\[
	\beta\colon \bS^{n-1}\to \bS^{n-1}\text{ s.t. }\beta(e)\perp e\quad \forall e\in \bS^{n-1},
\]
and an analogous statement holds if we replace $\bS^{n-1}$ by $\bR\bP^{n-1}$ on the right-hand side.

\subsubsection{Surjectivity of image}
	\label{sssec:surjectivity_of_image}
We also record the following elementary consequence for even $n$.
Suppose $\beta\colon \bS^{n-1}\to \bS^{n-1}$ is continuous, with the property that $\beta(e)\perp e$.
Then $\beta$ is surjective when its image is descended to $\bR\bP^{n-1}:=\bS^{n-1}/\pm 1$.
Indeed, suppose by contradiction that the two vectors $\pm e_n$ are never in the image of $\beta$.
Restricted to $\bS^{n-2}(e_n^{\perp})$ we can decompose $\beta(e)=\beta_1(e) + \beta_2(e)$ with $\beta_2(e)=\alpha(e)e_n$ and $\beta_1(e)\in e_n^{\perp}$.
Now by assumption $\beta_1(e)\neq 0$ and also $\ip{e,\beta_1(e)}=0$, so we can define
\[
	\wtilde{\beta}(e):=\frac{\beta_1(e)}{\norm{\beta_1(e)}}\colon \bS^{n-2}\left(e_n^{\perp}\right)\to \bS^{n-2}\left(e_n^\perp\right)
\]
which has the property that $\wtilde{\beta}(e)\perp e$.
But this is a contradiction to \autoref{sssec:the_simplest_obstruction}	since $n-2$ is even.

\subsubsection{Restrictions on reduction of structure group}
	\label{sssec:restrictions_on_reduction_of_structure_group}
We next recall some more elaborate topological facts from \cite[Proof of Thm.~3.1, Step 1]{CekicLefeuvreMoroianu2023_On-the-ergodicity-of-the-frame-flow-on-even-dimensional-manifolds}.
Consider the principal $\SO_{n-1}(\bR)$-bundle over the $(n-1)$-sphere:
\begin{equation}
	\label{eqn_cd:principal_sphere_bundle}
\begin{tikzcd}
	\SO_{n-1}(\bR)
	\arrow[r, hook, ""]
	&
	\SO_{n}(\bR)
	\arrow[d, ""]
	\\
	&
	\bS^{n-1}
\end{tikzcd}
\end{equation}
Suppose that this principal bundle has a continuous reduction to structure group $H\subseteq \SO_{n-1}(\bR)$, with $H$ connected.
Then there are the following possibilities for $H$, and except for $H=\SO_{n-1}(\bR)$, only these can occur:
\begin{description}
	\item[Typical cases] Suppose $n\neq 7, 8, 134$.
	Then:
	\begin{description}
		\item[$n$ odd] There does not exist a nontrivial reduction.
		\item[$n\equiv 2\mod 4$, or $n=4$] Then $H$ fixes a vector in $\bR^{n-1}$.
		\item[$n\equiv 0\mod 4$] Then $H$ acts reducibly on $\bR^{n-1}$.
	\end{description}
	\item[Exceptional cases] We have:
	\begin{description}
		\item[$n=7$] $H$ must contain $\SU(3)$ and fix a complex structure on $\bR^{n-1}=\bR^6$.
		\item[$n=8$] Either $H$ acts reducibly on $\bR^7$, or $H$ fixes a nonzero $3$-form on $\bR^7$ and $H$ equals either $\SO(3)$ or $G_2$.
		\item[$n=134$] Either $H$ fixes a vector in $\bR^{133}$, or it fixes a nonzero $3$-form and equals $E_7/(\bZ/2)$.
	\end{description}
\end{description}

\subsubsection{Full frame bundle}
	\label{sssec:full_frame_bundle}
We extend \autoref{eqn_cd:SM_FM_HM_M} to include the full frame bundle:
\begin{equation}
	\label{eqn_cd:full_frame_bundle}
	\begin{tikzcd}
		& PM \arrow[d] 	&
		\\
		&
		FM
		\arrow[dl, "\pi_U", swap]
		\arrow[dr, "\pi_H"]
		&
		\\
		UM
		\arrow[dr, "\pi_1", swap]
		&
		&
		HM
		\arrow[dl, "\pi_{n-1}"]
		\\
		&
		M
		&
	\end{tikzcd}
\end{equation}
where the various spaces involved can be described as:
\begin{align}
	\label{eqn:spaces_defined}
	\begin{split}
	PM(x)&:=\set{(e_1,\dots,e_n)\into T_x M\text{ an oriented frame}}\\
	UM(x)&:=\set{e_1\into T_x M\text{ unit vector}}\\
	HM(x)&:=\set{e_n\into T_x M\text{ unit vector}}\\
	FM(x)&:=\set{(e_1,e_n)\into T_x M\text{ a pair of unit vectors with }e_1\perp e_n }
	\end{split}
\end{align}
In particular, the hyperplane associated to $e_n\in HM(x)$ is defined to be $e_n^{\perp}\subset T_x M$.

Observe that $PM\to M$ is a principal $\SO_{n}(\bR)$-bundle over $M$, and a principal $\SO_{n-1}(\bR)$-bundle over $UM$.
Let $H\subset \SO_{n-1}(\bR)$ be a transitivity group of the cocycle $PM\to UM$.
Over a point $x\in M$ we have the diagram:
\begin{equation}
	\label{eqn_cd:principal_sphere_bundle_over_point}
	\begin{tikzcd}
		\SO_{n-1}(\bR)
		\arrow[r, hook, ""]
		&
		PM(x)
		\arrow[d, ""]
		\\
		&
		\bS^{n-1}(T_x M)
	\end{tikzcd}
	\end{equation}
Since the reduction of the structure group to $H$ is continuous, the restrictions from \autoref{sssec:restrictions_on_reduction_of_structure_group} apply.

Because we know that our $M$ contains totally geodesic hypersurfaces, we can further improve the information provided by \autoref{sssec:restrictions_on_reduction_of_structure_group}:

\begin{proposition}[Restrictions on the transitivity group]
	\label{prop:restrictions_on_the_transitivity_group}
	With assumptions as in \autoref{thm:main}, let $H\subset \SO_{n-1}(\bR)$ be the transitivity group of $PM\to UM$, which we assume is connected by \autoref{rmk:on_transitivity_groups}.
	Then we have the following possibilities for $H$ and coset space $C:=\leftrightquot{H}{\SO_{n-1}(\bR)}{\SO_{n-2}(\bR)}$:
	\begin{description}
		\item[$n$ odd] Then only the following possibilities can occur:
		\begin{enumerate}
			\item $H=\SO_{n-1}(\bR)$ and $C=\set{pt}$ (any $n$).
			\item $H=\SU(3)$ and $C=\set{pt}$ (for $n=7$ only).
		\end{enumerate}
		\item[$n$ even] Then we have three possibilities:
		\begin{enumerate}
			\item $H=\SO_{n-1}(\bR)$ and $C=\set{pt}$ (any $n$).
			\item $H=\SO_{n-2}(\bR)$ and $C=[-1,1]$ (any $n$).
			\item $H=\SU(3)$ and $C=[-1,1]$ (for $n=8$ only).
		\end{enumerate}
	\end{description}
	Above, ``any $n$'' means subject to $n\geq 3$ and the specified parity condition.
\end{proposition}
\begin{proof}
	In addition to the results enumerated in \autoref{sssec:restrictions_on_reduction_of_structure_group}, we will also make use of the fact that $M$ contains a closed, immersed, totally geodesic hypersurface $S\subset M$.
	Over $S$ we have the transitivity group $H_S\subset \SO_{n-2}(\bR)$ which satisfies the analogous set of restrictions, but we also have an inclusion $H_S\subset H_M$, since we visibly have such an inclusion at the level of Brin groups, and the Brin groups are isomorphic to the transitivity groups.
	Note that while the transitivity group of $S$ might not be connected, by \autoref{rmk:on_transitivity_groups} we can assume that the transitivity group of $M$ is.

	Suppose first that $n$ is odd.
	If $n\neq 7$ then the results enumerated in \autoref{sssec:restrictions_on_reduction_of_structure_group} imply the claim.
	If $n=7$, we have the possibility that $H=\SU(3)$ but $\SU(3)$ acts transitively on $\bS^5$ so the claim $C=\set{pt}$ follows.

	Suppose next that $n$ is even.
	Because a hypersurface $S$ is necessarily odd-dimensional, we have that $H_S=\SO_{n-2}(\bR)$ except possibly when $n=8$ in which case $H_S=\SU(3)$ is also possible.
	So for all even $n\neq 8$ we have the claim, and since the orbits of $\SO_{n-2}(\bR)$ on $\bS^{n-2}\subset \bR^{n-1}$ are classified by their last coordinate, an element of $[-1,1]$.
	For $n=8$ we also have the possibility that $H=H_S=\SU(3)$.
	But the corresponding action of $\SU(3)$ on $\bS^6$ has orbit space again equal to $[-1,1]$ and given by the value of the last coordinate.
\end{proof}

With the restrictions obtained in the preceding result, we can now establish our main goal:
\begin{proposition}[A totally geodesic through almost every hyperplane]
	\label{prop:a_totally_geodesic_through_almost_every_hyperplane}
	Denote by $Z(x)\subset HM(x)$ the intersection of the set $Z$ defined in \autoref{sssec:diagram_of_vectors_and_subspaces} with a fiber of $HM\to M$.
	Then for all $x\in M$ we have that $\dim Z(x) = \dim HM(x)=n-1$.
\end{proposition}
\begin{proof}
	From \autoref{eqn_cd:SM_FM_HM_M} it suffices to show that the set $I(x)\subset FM(x)$ has the same dimension as $FM(x)$, since $I=\pi_{H}^{-1}(Z)$.
	Recall also that $FM=PM/\SO_{n-2}(\bR)$ and we have an action of $G_t$ lifting the geodesic flow $g_t$ on $UM$.
	Let $H\subset \SO_{n-1}(\bR)$ be the transitivity group of $PM\to UM$.

	We know that:
	\begin{itemize}
		\item The set $I$ is closed, $G_t$-invariant, and projects surjectively to $UM$, hence by \autoref{prop:closed_invariant_sets} its generic part $I_{gen}$ corresponds to a nonempty closed subset of
		\[
			\leftquot{H}{\bS^{n-2}} = \leftrightquot{H}{\SO_{n-1}(\bR)}{\SO_{n-2}(\bR)}.
		\]
		\item The set $I$ is saturated by the fibers of the projection to $HM$, by definition.
	\end{itemize}
	From \autoref{prop:restrictions_on_the_transitivity_group}, only the cases when $C\neq \set{pt}$ need to be treated, i.e. $C=[-1,1]$ and $n$ is even.


	We next observe that for $n$ even, regardless of whether $n=8$ or not, the datum of the Brin group $B(u)$ gives a family of maps $\beta(x,-)\colon UM(x)\to HM(x)$ with the property that $\beta(x,e_1)\perp e_1$ for all $e_1\in UM(x)$.
	Indeed, the Brin group is contained in the stabilizer of $\beta(x,e_1)$.
	Now the classifying map is given by: 
	\begin{align*}
		cl \colon FM &\to [-1,1]\\
		(x,e_1, e_n) & \mapsto \ip{\beta(x,e_1),e_n}
	\end{align*}
	Observe that in the case $n=8$, we can lift $\beta$ to a higher coset space, but we do not need to do so for the argument.
	Recall that $I_{gen}\subset I$ is the generic part.
	Then there are two possibilities analyzed below:
	\emph{Case 1:} $cl(I_{gen})\subset \set{\pm 1}$ and \emph{Case 2:} $cl(I_{gen})\ni \alpha \in (-1,1)$.
	Fix some $x\in M$ for the rest of the argument and omit it from the notation; we will write $\beta(e_1)$ instead of $\beta(x,e_1)$, and $\bS^{n-1}$ will refer to the unit sphere in $T_x M$.

	We also recall from \autoref{sssec:surjectivity_of_image} that $\beta$ is surjective onto $\bR\bP^{n-1}$.
 Note that the sets $I $ and $Z$ are invariant under the involution which switches the orientation of the hyperplane (i.e. $(e_1,e_n)\mapsto (e_1,-e_n)$).
 Therefore, $I_{gen}$ is also invariant under this involution, since it commutes with the flow $G_t$ and the fiber over a generic point is invariant under this involution.

	\noindent\textbf{Case 1:}
	This means that $\ip{\beta(e_1),e_n}=\pm 1, \forall (e_1,e_n)\in I_{gen}$, i.e. $\beta(e_1)=\pm e_n$ for any pair $(e_1,e_n)\in I_{gen}$.
	Since $\beta$ is surjective onto $\bR\bP^{n-1}$  and $Z$ is invariant under the orientation-reversing involution $e_n\mapsto -e_n$, it follows that $Z$ contains an open set and our needed claim is settled.

	\noindent\textbf{Case 2:}
	Now suppose that $\alpha\in (-1,1)$ is in $cl(I_{gen})$.
	This means that for any $e_1\in \bS^{n-1}$, we have that $I_{gen}$ (and thus $I$) contains the set
	\[
		\set{(e_1,e_n)\colon \ip{\beta(e_1),e_n}=\alpha}
	\]
	For a fixed $e_1$, the set of $e_n$'s that satisfy the above condition form the boundary of a spherical cap, diffeomorphic to $\bS^{n-2}$ and described by the formulas:
	\[
		e_n = \alpha\cdot \beta(e_1) + e_n' \quad \text{with } e_n'\in \beta(e_1)^\perp, \norm{e_n'}^2 = 1-\alpha^2.
	\]
	Let us denote this set by $\bS^{n-2}_{\alpha}(e_1)\subset \bS^{n-1}$.

	Take now two distinct values $\beta(e_1')\neq \beta(e_1'')$.
	Then the intersection
	\[
		L:=
		\set{e_n\colon \ip{\beta(e_1'),e_n}>\alpha}
		\cap
		\set{e_n\colon \ip{\beta(e_1''),e_n}<\alpha}
	\]
	is a nonempty open set, as the intersection of two open spherical caps around $\beta(e_1')$ and $-\beta(e_1'')$.
 Non-emptiness holds because the two spherical caps have measure adding up to the measure of $\bS^{n-1}$, and the only situation in which they are disjoint is when $\beta(e_1')=\beta(e_1'')$.
	For a continuous path $e_1(t)$ with $e_1(0)=e_1',e_1(1)=e_1''$, we have that any element $e_n\in L$ is in some $\bS^{n-2}_{\alpha}(e_1(t))$, since $f(t):=\ip{\beta(e_1(t),e_n)}$ satisfies $f(0)>\alpha,f(1)<\alpha$, so by the intermediate value theorem there exists $t$ such that $f(t)=\alpha$.
	This shows that an open set of hyperplanes belongs to $Z$ and concludes the proof.
\end{proof}



\subsubsection{Last step}
	\label{sssec:real_hyperbolic_last_step}
We know that $Z\subset HM$ is a closed real-analytic set of full dimension, therefore it is in fact all of $HM$.
Indeed, for any point of $z\in Z$, the germ of $Z$ at $z$ denoted $Z_z$ has full dimension, therefore by \cite[Ch.~V, Prop.~1, pg.~91]{Narasimhan1966_Introduction-to-the-theory-of-analytic-spaces} the complexified germ $\wtilde{Z}_z$ is an entire complex neighborhood of $z$, and hence the germ $Z_z$ is an entire real neighborhood of $z$ by the result just cited.
It follows that $Z$ is also open, hence equal to $HM$.

It now follows from ``Cartan's axioms of $k$-planes'' (with $k=n-1$) \autoref{thm:Cartan}, that $M$ has constant sectional curvature.
Since the geodesic flow is Anosov, the constant curvature must be negative.



\subsection{Proof of \autoref{thm:unclosed}}

We briefly describe the modifications required for the proof of \autoref{thm:unclosed}.
The first main point is to show that the set $Z$ must still have dimension $n$ as in the beginning of
\autoref{sssec:structure_of_real_analytic_sets}. Let $N$ be the
non-closed totally geodesic submanifold and consider it's lift $\tilde{N}$ to $HM$.  This
is an embedded submanifold since a tangent hyperplane determines the totally geodesic manifold
as before.  Consider any metric on $HM$ making the projection a Riemannian submersion.  Any point in $\tilde{N}$ has an open neighborhood $U$ in $\tilde{N}$ of radius at least the injectivity radius of $M$.  Thus $\tilde{N}$ is not locally closed and so the stratum of $Z$ containing $\tilde{N}$ must have dimension at least $n$.

It remains to see how we can restrict the transitivity group so as to argue as above.  If  $\dim(M)$ is odd, then the argument is exactly as in \autoref{prop:reduction_of_structure_group}.  If the dimension is even
and the curvature is pinched as in \cite[Theorem 1.2.]{CekicLefeuvreMoroianu2023_On-the-ergodicity-of-the-frame-flow-on-even-dimensional-manifolds} then the proof of that theorem shows $H=SO(n-1)$.
If $n$ is even and not divisible by $4$ and so equals $4k+2$ then following the proof of \cite[Theorem 3.8]{CekicLefeuvreMoroianu2023_On-the-ergodicity-of-the-frame-flow-on-even-dimensional-manifolds}, we see that unless $n=134$, we have $H<SO(n-2)$. By \cite[Theorem 3]{Ozaki} we have that $H$ can only be $SO(n-2), U(n-2)$ or $SU(n-2)$ unless $n=10$ or $n=18$. In all such cases we have that $C$ is $[-1,1]$ and we can argue as above.

In dimension $10$ we can in addition have $H=Spin(7)<SO(8)$.  By work of Adams, the $9$ sphere only admits one non-vanishing vector field and so the resulting $8$ dimensional representation of $Spin(7)$ cannot have invariant vectors. Therefore the representation is the spin representation of $Spin(7)$ \cite{Adams}.  There is a maximal $G_2$ in $Spin(7)$ for which this representation splits off a single trivial representation realizing $Spin(7)/G_2$ as $S^7$, the unit tangent sphere in our $8$ dimensional representation see e.g. the discussion of maximal subgroups in \cite{Bourbaki} and of branching rules in \cite{MckayPatera}.  It follows from this that $C$ is once again $[-1,1]$ and we argue as before.

In dimension $18$, we can in addition have $H=Spin(9)<SO(16)$.  As before, by work of Adams, this is an irreducible representation and so the spin representation of $Spin(9)$.  The restriction of this representation to
$Spin(8)$ is the sum of the two semi-spin representations of $Spin(8)$.  In $Spin(8)$ there are three realizations of $Spin(7)$ coming from triality and for two of these copies of $Spin(7)$ the $16$ dimensional spin representation of $Spin(9)$ splits as a trivial representation, and two irreducibles of dimension $7$ and $8$.  See once again \cite{Bourbaki, MckayPatera}.  This shows that there is a $Spin(7)$ invariant vector in the $16$ dimensional representation of $Spin(9)$ and realizes the orbit of that vector as $S^{15}$.  This is once again exactly what we need to see that $C$ is $[-1,1]$.

Finally if we are in even dimensions and there is also one closed immersed totally geodesic submanifold $N'$ in $M$ then we once again see that the transitivity group is either $SO(n-1)$ or $SO(n-2)$ and
the proof works as before.




\section{Examples and Further Directions}
	\label{sec:examples_and_further_directions}


\subsection{Examples}
	\label{ssec:examples}

This section collects examples to illustrate some key points concerning our results.  The first examples
show that there is no analogue of our theorem in positive curvature even for analytic metrics.  The second collection of examples show that one cannot prove an analogous theorem on a simply connected manifold of negative curvature.  Together these classes of examples emphasize that the role played by dynamical recurrence in our proofs is not just an artifact of the techniques.  Finally we give  examples of closed analytic Riemannian manifolds with fixed topology, non-constant negative curvature, and an arbitrarily large but finite number of closed immersed totally geodesic hypersurfaces.

\subsubsection{Examples in positive curvature}
	\label{sssec:examples_in_positive_curvature}
First we construct examples of positively curved, real-analytic Riemannian manifolds that have infinitely many totally geodesic hypersurfaces, but that are distinct from the round sphere.
Consider the Riemannian metric on $\bS^n$ given by: 
\begin{equation} \label{positivezoll}
	g_{h,\bS^n}:=\left[1+ h(\sin(\phi)\right]^2 d\phi^2 + \cos(\phi)^2g_{0,\bS^{n-1}}
\end{equation}
where $h\colon (-1-\ve,1+\ve)\to \bR$ is an even, real-analytic function that vanishes at $-1,1$.
Note that $h\equiv 0$ yields the standard warped product formula for the round metric, and for $h$ small enough (in the $C^2$-sense) the metric will still have positive sectional curvature.

The metric also admits an order $2$ isometry $\sigma(x,\phi):= (\sigma' x,\phi)$, where $\sigma'\colon \bS^{n-1}\to \bS^{n-1}$ is a reflection fixing a copy of $\bS^{n-2}$.
The fixed point set of $\sigma$ is a totally geodesic hypersurface.
Rotating it in the $\bS^{n-1}$-factor gives infinitely many such hypersurfaces.


\subsubsection{Examples in negative curvature}
	\label{sssec:examples_in_negative_curvature}
The next examples show that there is no local obstruction to a simply connected manifold of non-constant negative curvature having infinitely many totally geodesic hypersurfaces. First, there are many examples of homogeneous variably negatively curved Riemannian manifolds that have infinitely many totally geodesic hypersurfaces. These Riemannian manifolds are isometric to solvable Lie groups equipped with left invariant Riemannian metrics \cite{heintze1974homogeneous}.  Lin--Schmidt  gave examples of simply connected non-homogeneous negatively curved 3-manifolds each tangent vector to which is contained in a totally geodesic surface \cite{ls17}.  In fact, these totally geodesic surfaces can be taken to be isometric to the hyperbolic plane.

It is also possible to construct a negative curvature analogue of the examples of (\ref{positivezoll}) as follows.  Writing the hyperbolic metric on $\mathbb{H}^{n+1}$ in polar coordinates as $\sinh^2(r) g_{\bS^n}(\theta) + dr^2$, for $g_{\bS^n}$ the round metric on $\bS^n$, we can define a new family of metrics by
\[
g_{h,\mathbb{H}^{n+1}} = (\sinh^2 r)  g_{ \varphi(r) h,\bS^n}(\theta)  + dr^2,
\]
for $h$ an even real analytic function as above, $\varphi:\mathbb{R} \rightarrow \mathbb{R}^{\geq 0}$ a smooth monotone non-decreasing function equal to $0$ in some neighborhood of $0$ and $1$ for all large enough $r$, and $g_{h,\bS^n}$ the metric on $\bS^{n}$ defined in \autoref{positivezoll}. For $h\equiv 0$ the $\theta$ variable corresponds to the angular coordinate, and $r$ is equal to the distance to the origin. For any fixed $R>0$, $g_{h,\mathbb{H}^{n+1}}$ can be chosen to be as $C^2$-close as desired to the hyperbolic metric for $r<R$, provided that $h$ was chosen with sufficiently small $C^2$-norm and $\varphi$ was chosen to have small enough first and second derivatives.  The isometries of $g_{\varphi(r)h,\bS^n}$ described above extend to isometries of $g_{h,\mathbb{H}^{n+1}}$ the fixed point sets of which are totally geodesic hypersurfaces.

 Furthermore, provided $g_{h,\bS^n}$ is close enough to the round metric and $\varphi$ is chosen with small enough first and second derivatives, $g_{h,\mathbb{H}^{n+1}}$ will have negative curvature.  One can check this as follows. For points with $r$-coordinate larger than some $R$ depending only on the $C^2$ norm of $h$ and $\varphi$, all of their tangent planes will have negative sectional curvature by the formulas that compute the curvature of warped product metrics see \cite[Lemma~3.5]{FJ} and compare  \cite[Lemma~6.2]{l21}.  The remaining points with $r$ coordinate less than $R$ will then also have negative curvature, provided $g_{h,\bS^n}$ was chosen sufficiently close to the round metric and $\varphi$ was chosen with sufficiently small first and second derivative, as observed in the previous paragraph.




\subsubsection{Closed negatively curved examples with finitely many totally geodesic hypersurfaces}
	\label{sssec:closed_negatively_curved_examples_with_finitely_many_totally_geodesic_hypersurfaces}
It is possible to construct \emph{smooth} Riemannian metrics of negative curvature on a closed hyperbolizable $n$-manifold, $n>2$, that are arbitrarily $C^2$-close to the hyperbolic metric and that contain an arbitrarily large number of totally geodesic hypersurfaces.
This can be done provided that there are infinitely many closed totally geodesic hypersurfaces in the hyperbolic metric. One constructs such metrics $g_{\epsilon}$ by choosing a finite union of closed totally geodesic hypersurfaces in the hyperbolic (constant sectional curvature -1) metric, and then perturbing the metric on a small ball disjoint from the hypersurfaces in the finite union.  We remark that it is sometimes even possible to construct such examples while keeping the sectional curvature bounded above by $-1$ and the totally geodesic hypersurfaces hyperbolic (constant sectional curvature -1 in their induced metric), see \cite[Section 6]{Lowe}.

We can construct analytic metrics $g_{\epsilon}$ as in the previous paragraph as follows.  Suppose that $M$ contains infinitely many totally geodesic hypersurfaces in its hyperbolic metric each of which lifts to a finite cover where it is the fixed point set of a finite group of isometries of that finite cover. All first type arithmetic hyperbolic manifolds have infinitely many totally geodesic hypersurfaces of this kind, see \cite{belolipetsky2021subspace}. If we want $k$ totally geodesic hypersurfaces, we choose $g_{\epsilon}$ as above by modifying $g$ away from the $k$ totally geodesic hypersurfaces. We then run Ricci flow for a short time on the metrics $g_{\epsilon}$ of the previous paragraph to obtain metrics $g_{\epsilon}(t)$.

Note that $g_{\epsilon}(t)$ is analytic for $t>0$ by the main result of \cite{Bando}.  Also, since the Ricci flow preserves isometries, and since the fixed point set of a group of isometries is totally geodesic, by pulling back the Ricci flow to the finite covers described in the previous paragraph one can check that each of the $k$ totally geodesic hypersurfaces in $g_{\epsilon}$ is homotopic to a nearby totally geodesic hypersurface in $g_{\epsilon}(t)$.

In fact, we can modify the above construction to give analytic metrics with \textit{exactly} k totally geodesic hypersurfaces. We write this as a proposition.

\begin{proposition} \label{prop:exactlyk}
For each positive integer $k$ there are closed analytic Riemannian manifolds $M_k$ of non-constant negative curvature that contain exactly $k$  closed maximal totally geodesic hypersurfaces. Moreover, the $M_k$ can be chosen to be diffeomorphic.
\end{proposition}

\begin{proof}
For one of the metrics $g_{\epsilon}$ as above, we can perturb $g_{\epsilon}$ away from a small neighborhood $U$ of the k totally geodesic hypersurfaces so that no totally geodesic hypersurface in $(M,g_{\epsilon})$ passes through a point of $M-U$. This is possible by using the deformations considered by (\cite{MurphyWilhelm}, \cite{ElHWilhelm}), which show that a condition on the curvature tensor at point, that is open on the set of Riemannian metrics in the $C^2$-topology, prevents a totally geodesic hypersurface from passing through that point.

Call this perturbation $g_{\epsilon}'$.  Then, running Ricci flow on $g_{\epsilon}'$ for a short time, we obtain metrics  $g_{\epsilon}'(t)$ that have $k$ totally geodesic hypersurfaces homotopic to the $\Sigma_1,..,\Sigma_k$ as above. If we chose $t$ small enough, it will be the case that no point of $M-U'$ has a totally geodesic hypersurface in the metric $g_{\epsilon}'(t)$ passing through it, for $U'$ a slightly larger open neighborhood than $U$ of the union of the $\Sigma_i$. Here we are  using the fact that the condition on the curvature tensor that (\cite{MurphyWilhelm}, \cite{ElHWilhelm}) used to check that the perturbed metrics they constructed contained no totally geodesic hypersurfaces was open in the $C^2$-topology on the space of Riemannian metrics. Thus, any  closed totally geodesic hypersurface $\Sigma$ in $(M,g_{\epsilon}'(t))$ must be contained in $U'$.  Each tangent hyperplane to $\Sigma$ must then be at a distance of at most $\delta$ in the Grassman bundle $Gr_{n-1}(M,g_{\epsilon}(t)')$ from a tangent hyperplane to one of the $\Sigma_i$, where $\delta$ can be made as small as desired making $U$, $U'$, and $t$ small enough.  If $\delta$ is chosen small enough, one can then show that normal projection from $\Sigma$ to one of the $\Sigma_i$ defines a covering map.This shows that $(M,g_{\epsilon}'(t))$ contains exactly $k$ closed maximal totally geodesic hypersurfaces.
\end{proof}

It seems possible that one could also do a similar construction in the context of the examples of Gromov and Thurston \cite{GromovThurston}  and build metrics with at least $k$ totally geodesic submanifolds on negatively curved manifolds which do not admit metrics of constant negative curvature.

Finally we recall that given a manifold $M$ and an embedded hypersurface $N$ it is well known that one can define a metric $g$ on $M$ in which $N$ is totally geodesic and which contains infinitely many totally geodesic submanifolds diffeomorphic to $N$.  This is done by taking a product metric on a tubular neighborhood of $N$ and extending smoothly to get a metric on $M$.  Clearly the same can be done for any finite collection of disjoint embedded submanifolds.
We do not know if there is a general analytic version of this construction.



\subsection{Directions for Future Work and Conjectures}
	\label{ssec:directions_for_future_work_and_conjectures}

As in \cite{BFMS1, BFMS2} given a manifold $M$, we say that an immersed totally geodesic submanifold
$N$ in $M$ is \emph{maximal} if it is not contained in another proper closed immersed totally geodesic submanifold.  It is then natural to conjecture the following generalization of \autoref{thm:main}:

\begin{conjecture}
\label{conj:tggen}
	Let $(M,g)$ be a closed Riemannian manifold with negative sectional curvature, of dimension $n\geq 3$.
	Suppose that $M$ contains infinitely many maximal closed totally geodesic immersed submanifolds of dimension at least $2$. Then $M$ is a locally symmetric space of rank $1$.
\end{conjecture}

\noindent If the conjecture is true, it then follows from the results of \cite{BFMS1, BFMS2, Corlette, GromovSchoen} that $M$ is arithmetic as in \autoref{thm:main}.  We do not specify the regularity in the conjecture, it seems possible that it holds as soon as the metric is $C^2$ and so all the notions in the conjecture are well defined.

The work in \cite{BFMS1, BFMS2} was motivated in part by the Margulis commensurator arithmeticity theorem.
In that context, there is an analogy between \autoref{conj:tggen} and results by Eberlein and Farb--Weinberger \cite{Eberlein1,FarbWeinberger}.  In particular, the results of Farb--Weinberger suggest there might be an analogue of \autoref{conj:tggen} in the broader context of Riemannian metrics
on closed aspherical manifolds.

There are other more general questions than \autoref{conj:tggen} if one thinks of \autoref{thm:main} as a statement
about closed manifolds invariant under Anosov flows.  It seems possible that an Anosov flow admitting infinitely many  closed, flow invariant $C^1$ submanifolds is necessarily algebraically defined.
Related questions have been considered by Zeghib \cite{Zeghib1995_Sur-une-notion-dautonomie-de-systemes-dynamiques-appliquee-aux-ensembles-invariants}.

One can also conjecture a generalization of \autoref{thm:unclosed}:

\begin{conjecture}
\label{conj:unclosedgen}
	Let $(M,g)$ be a closed Riemannian manifold with negative sectional curvature, of dimension $n\geq 3$.
	Suppose that $M$ contains a maximal totally geodesic immersed submanifold $N$ which is not closed.
	Then the closure of $N$ is an immersion of a totally geodesic locally symmetric space of rank $1$ in $M$.
\end{conjecture}



\bibliographystyle{sfilip_bibstyle}
\bibliography{tot_geod_hyp}

@book{Pesin2004_Lectures-on-partial-hyperbolicity-and-stable-ergodicity,
    Title = {Lectures on partial hyperbolicity and stable ergodicity},
    Publisher = {European Mathematical Society (EMS), Z\"urich},
    Year = {2004},
    Author = {Pesin, Yakov B.},
    Series = {Zurich Lectures in Advanced Mathematics},
    url = {https://doi.org/10.4171/003},
    doi = {10.4171/003},
    pages = {vi+122},
}

@article{Zeghib_Laminations-et-hypersurfaces-geodesiques-des-varietes1991,
    Author = {Zeghib, A.},
    Title = {Laminations et hypersurfaces g\'{e}od\'{e}siques des vari\'{e}t\'{e}s hyperboliques},
    Journal = {Ann. Sci. \'{E}cole Norm. Sup. (4)},
    Year = {1991},
    Volume = {24},
    Number = {2},
    Pages = {171--188},
    url = {http://www.numdam.org/item?id=ASENS_1991_4_24_2_171_0},
    pages = {171--188},
}

@book{BallmannGromovSchroeder1985_Manifolds-of-nonpositive-curvature,
    Title = {Manifolds of nonpositive curvature},
    Publisher = {Birkh\"{a}user Boston, Inc., Boston, MA},
    Year = {1985},
    Author = {Ballmann, Werner and Gromov, Mikhael and Schroeder, Viktor},
    Volume = {61},
    Series = {Progress in Mathematics},
    url = {https://doi.org/10.1007/978-1-4684-9159-3},
    doi = {10.1007/978-1-4684-9159-3},
    pages = {vi+263},
}

@article{FisherLafont_Finiteness-of-maximal-geodesic-submanifolds2021,
    Author = {Fisher, David and Lafont, Jean-Fran\c{c}ois and Miller, Nicholas and Stover, Matthew},
    Title = {Finiteness of maximal geodesic submanifolds in hyperbolic hybrids},
    Journal = {J. Eur. Math. Soc. (JEMS)},
    Year = {2021},
    Volume = {23},
    Number = {11},
    Pages = {3591--3623},
    url = {https://doi.org/10.4171/jems/1077},
    doi = {10.4171/jems/1077},
    pages = {3591--3623},
}

@book{Margulis_Discrete-subgroups-of-semisimple-Lie-groups1991,
    Title = {Discrete subgroups of semisimple {L}ie groups},
    Publisher = {Springer-Verlag, Berlin},
    Year = {1991},
    Author = {Margulis, G. A.},
    Volume = {17},
    Series = {Ergebnisse der Mathematik und ihrer Grenzgebiete (3) [Results in Mathematics and Related Areas (3)]},
    url = {https://doi.org/10.1007/978-3-642-51445-6},
    doi = {10.1007/978-3-642-51445-6},
    pages = {x+388},
}

@article{Eberlein1973_When-is-a-geodesic-flow-of-Anosov-type-III,
    Author = {Eberlein, Patrick},
    Title = {When is a geodesic flow of {A}nosov type? {I},{II}},
    Journal = {J. Differential Geometry},
    Year = {1973},
    Volume = {8},
    Pages = {437--463; ibid. {\bf 8 (1973), 565--577}},
    url = {http://projecteuclid.org/euclid.jdg/1214431801}
}

@article{Zeghib1995_Sur-une-notion-dautonomie-de-systemes-dynamiques-appliquee-aux-ensembles-invariants,
    Author = {Zeghib, A.},
    Title = {Sur une notion d'autonomie de syst\`emes dynamiques, appliqu\'{e}e aux ensembles invariants des flots d'{A}nosov alg\'{e}briques},
    Journal = {Ergodic Theory Dynam. Systems},
    Year = {1995},
    Volume = {15},
    Number = {1},
    Pages = {175--207},
    url = {https://doi.org/10.1017/S0143385700008300},
    doi = {10.1017/S0143385700008300},
    pages = {175--207},
}

@article{BrinGromov1980_On-the-ergodicity-of-frame-flows,
    Author = {Brin, M. and Gromov, M.},
    Title = {On the ergodicity of frame flows},
    Journal = {Invent. Math.},
    Year = {1980},
    Volume = {60},
    Number = {1},
    Pages = {1--7},
    url = {https://doi.org/10.1007/BF01389897},
    doi = {10.1007/BF01389897},
    pages = {1--7},
}

@article{EskinFilipWright2018_The-algebraic-hull-of-the-Kontsevich-Zorich-cocycle,
    Author = {Eskin, Alex and Filip, Simion and Wright, Alex},
    Title = {The algebraic hull of the {K}ontsevich-{Z}orich cocycle},
    Journal = {Ann. of Math. (2)},
    Year = {2018},
    Volume = {188},
    Number = {1},
    Pages = {281--313},
    url = {https://doi.org/10.4007/annals.2018.188.1.5},
    doi = {10.4007/annals.2018.188.1.5},
    pages = {281--313},
}

@article{Filip_Translation-surfaces:-Dynamics-and-Hodge-theory,
    Author = {{Filip}, Simion},
    Title = {Translation surfaces: Dynamics and Hodge theory},
    Journal = {EMS Surv. Math. Sci.},
    Year = {2024},
    Volume = {(to appear)},
    Pages = {1-91},
    pages = {1-91},
}

@article {MMO2,
    AUTHOR = {McMullen, Curtis T. and Mohammadi, Amir and Oh, Hee},
     TITLE = {Geodesic planes in the convex core of an acylindrical
              3-manifold},
   JOURNAL = {Duke Math. J.},
  FJOURNAL = {Duke Mathematical Journal},
    VOLUME = {171},
      YEAR = {2022},
    NUMBER = {5},
     PAGES = {1029--1060},
      ISSN = {0012-7094},
   MRCLASS = {57K32 (22E40 37A17)},
  MRNUMBER = {4402699},
MRREVIEWER = {Boris N. Apanasov},
       DOI = {10.1215/00127094-2021-0030},
       URL = {https://doi.org/10.1215/00127094-2021-0030},
}

@article {MMO1,
    AUTHOR = {McMullen, Curtis T. and Mohammadi, Amir and Oh, Hee},
     TITLE = {Geodesic planes in hyperbolic 3-manifolds},
   JOURNAL = {Invent. Math.},
  FJOURNAL = {Inventiones Mathematicae},
    VOLUME = {209},
      YEAR = {2017},
    NUMBER = {2},
     PAGES = {425--461},
      ISSN = {0020-9910},
   MRCLASS = {30F30 (30F35 37D40 37F30 53C40 57M50)},
  MRNUMBER = {3674219},
MRREVIEWER = {Rafael Oswaldo Ruggiero},
       DOI = {10.1007/s00222-016-0711-3},
       URL = {https://doi.org/10.1007/s00222-016-0711-3},
}

@article {Eberlein1,
    AUTHOR = {Eberlein, Patrick},
     TITLE = {Isometry groups of simply connected manifolds of nonpositive
              curvature. {II}},
   JOURNAL = {Acta Math.},
  FJOURNAL = {Acta Mathematica},
    VOLUME = {149},
      YEAR = {1982},
    NUMBER = {1-2},
     PAGES = {41--69},
      ISSN = {0001-5962},
   MRCLASS = {53C20 (53C30)},
  MRNUMBER = {674166},
MRREVIEWER = {O. Kowalski},
       DOI = {10.1007/BF02392349},
       URL = {https://doi.org/10.1007/BF02392349},
}

@article {FarbWeinberger,
    AUTHOR = {Farb, Benson and Weinberger, Shmuel},
     TITLE = {Isometries, rigidity and universal covers},
   JOURNAL = {Ann. of Math. (2)},
  FJOURNAL = {Annals of Mathematics. Second Series},
    VOLUME = {168},
      YEAR = {2008},
    NUMBER = {3},
     PAGES = {915--940},
      ISSN = {0003-486X},
   MRCLASS = {53C24 (53C35 57S25)},
  MRNUMBER = {2456886},
MRREVIEWER = {Alain Valette},
       DOI = {10.4007/annals.2008.168.915},
       URL = {https://doi.org/10.4007/annals.2008.168.915},
}

@book {BallmannLectures,
    AUTHOR = {Ballmann, Werner},
     TITLE = {Lectures on spaces of nonpositive curvature},
    SERIES = {DMV Seminar},
    VOLUME = {25},
      NOTE = {With an appendix by Misha Brin},
 PUBLISHER = {Birkh\"{a}user Verlag, Basel},
      YEAR = {1995},
     PAGES = {viii+112},
      ISBN = {3-7643-5242-6},
   MRCLASS = {53C21 (58F17)},
  MRNUMBER = {1377265},
MRREVIEWER = {Boris Hasselblatt},
       DOI = {10.1007/978-3-0348-9240-7},
       URL = {https://doi.org/10.1007/978-3-0348-9240-7},
}

@article{l21,
  title={Deformations of totally geodesic foliations and minimal surfaces in negatively curved 3-manifolds},
  author={Lowe, Ben},
  journal={Geometric and Functional Analysis},
  volume={31},
  number={4},
  pages={895--929},
  year={2021},
  publisher={Springer}
}

@article{riemann1854uber,
  title={Uber die Hypothesen, welche der Geometrie zu Grunde liegen},
  author={Riemann, Bernhard},
  journal={K{\"o}nigliche Gesellschaft der Wissenschaften und der Georg-Augustus-Universit{\"a}t G{\"o}ttingen},
  volume={13},
  number={133},
  pages={1867},
  year={1854}
}

@article{Brin1975_The-topology-of-group-extensions-of-C-systems,
    Author = {Brin, M. I.},
    Title = {The topology of group extensions of {$C$}-systems},
    Journal = {Mat. Zametki},
    Year = {1975},
    Volume = {18},
    Number = {3},
    Pages = {453--465},
    pages = {453--465},
}

@article{BarreiraValls2008_Analytic-invariant-manifolds-for-sequences-of-diffeomorphisms,
    Author = {Barreira, Luis and Valls, Claudia},
    Title = {Analytic invariant manifolds for sequences of diffeomorphisms},
    Journal = {J. Differential Equations},
    Year = {2008},
    Volume = {245},
    Number = {1},
    Pages = {80--101},
    url = {https://doi.org/10.1016/j.jde.2008.03.025},
    doi = {10.1016/j.jde.2008.03.025},
    pages = {80--101},
}

@book{KatokHasselblatt1995_Introduction-to-the-modern-theory-of-dynamical-systems,
    Title = {Introduction to the modern theory of dynamical systems},
    Publisher = {Cambridge University Press, Cambridge},
    Year = {1995},
    Author = {Katok, Anatole and Hasselblatt, Boris},
    Volume = {54},
    Series = {Encyclopedia of Mathematics and its Applications},
    url = {https://doi.org/10.1017/CBO9780511809187},
    doi = {10.1017/CBO9780511809187},
    pages = {xviii+802},
}

@article {Poincare,
    AUTHOR = {Poincare, H.},
     TITLE = {Sur les {E}quations {L}ineaires aux {D}ifferentielles
              {O}rdinaires et aux {D}ifferences {F}inies},
   JOURNAL = {Amer. J. Math.},
  FJOURNAL = {American Journal of Mathematics},
    VOLUME = {7},
      YEAR = {1885},
    NUMBER = {3},
     PAGES = {203--258},
      ISSN = {0002-9327},
   MRCLASS = {DML},
  MRNUMBER = {1505385},
       DOI = {10.2307/2369270},
       URL = {https://doi.org/10.2307/2369270},
}

@article {Brouwer,
    AUTHOR = {Brouwer, L. E. J.},
     TITLE = {\"{U}ber {A}bbildung von {M}annigfaltigkeiten},
   JOURNAL = {Math. Ann.},
  FJOURNAL = {Mathematische Annalen},
    VOLUME = {71},
      YEAR = {1912},
    NUMBER = {4},
     PAGES = {598},
      ISSN = {0025-5831},
   MRCLASS = {DML},
  MRNUMBER = {1511678},
       DOI = {10.1007/BF01456812},
       URL = {https://doi.org/10.1007/BF01456812},
}

@article {Ozaki,
    AUTHOR = {Ozaki, Yasuyuki},
     TITLE = {{$G$}-structures on spheres},
   JOURNAL = {Math. J. Okayama Univ.},
  FJOURNAL = {Mathematical Journal of Okayama University},
    VOLUME = {33},
      YEAR = {1991},
     PAGES = {189--200},
      ISSN = {0030-1566},
   MRCLASS = {55R10},
  MRNUMBER = {1159895},
MRREVIEWER = {Donald M. Davis},
}

@article {Eberlein,
    AUTHOR = {Eberlein, Patrick},
     TITLE = {Geodesic flows on negatively curved manifolds. {I}},
   JOURNAL = {Ann. of Math. (2)},
  FJOURNAL = {Annals of Mathematics. Second Series},
    VOLUME = {95},
      YEAR = {1972},
     PAGES = {492--510},
      ISSN = {0003-486X},
   MRCLASS = {58F15},
  MRNUMBER = {310926},
MRREVIEWER = {Leon\ W.\ Green},
       DOI = {10.2307/1970869},
       URL = {https://doi.org/10.2307/1970869},
}

@article {Bonatti,
    AUTHOR = {Bonatti, Christian},
     TITLE = {Champs de vecteurs analytiques commutants, en dimension {$3$}
              ou {$4$}: existence de z\'{e}ros communs},
   JOURNAL = {Bol. Soc. Brasil. Mat. (N.S.)},
  FJOURNAL = {Boletim da Sociedade Brasileira de Matem\'{a}tica. Nova S\'{e}rie},
    VOLUME = {22},
      YEAR = {1992},
    NUMBER = {2},
     PAGES = {215--247},
      ISSN = {0100-3569},
   MRCLASS = {57R25},
  MRNUMBER = {1179486},
MRREVIEWER = {P. Molino},
       DOI = {10.1007/BF01232943},
       URL = {https://doi-org.proxyiub.uits.iu.edu/10.1007/BF01232943},
}

@article {FarbShalen,
    AUTHOR = {Farb, Benson and Shalen, Peter},
     TITLE = {Real-analytic actions of lattices},
   JOURNAL = {Invent. Math.},
  FJOURNAL = {Inventiones Mathematicae},
    VOLUME = {135},
      YEAR = {1999},
    NUMBER = {2},
     PAGES = {273--296},
      ISSN = {0020-9910},
   MRCLASS = {22F30 (22E40 37C85 57S20)},
  MRNUMBER = {1666834},
MRREVIEWER = {Raul Quiroga-Barranco},
       DOI = {10.1007/s002220050286},
       URL = {https://doi-org.proxyiub.uits.iu.edu/10.1007/s002220050286},
}

@article {Adams,
    AUTHOR = {Adams, J. F.},
     TITLE = {Vector fields on spheres},
   JOURNAL = {Ann. of Math. (2)},
  FJOURNAL = {Annals of Mathematics. Second Series},
    VOLUME = {75},
      YEAR = {1962},
     PAGES = {603--632},
      ISSN = {0003-486X},
   MRCLASS = {57.30},
  MRNUMBER = {139178},
MRREVIEWER = {M. F. Atiyah},
       DOI = {10.2307/1970213},
       URL = {https://doi-org.proxyiub.uits.iu.edu/10.2307/1970213},
}

@book {Bourbaki,
    AUTHOR = {Bourbaki, N.},
     TITLE = {\'{E}l\'{e}ments de math\'{e}matique. {F}asc. {XXXIV}. {G}roupes et
              alg\`ebres de {L}ie. {C}hapitre {IV}: {G}roupes de {C}oxeter et
              syst\`emes de {T}its. {C}hapitre {V}: {G}roupes engendr\'{e}s par
              des r\'{e}flexions. {C}hapitre {VI}: syst\`emes de racines},
    SERIES = {Actualit\'{e}s Scientifiques et Industrielles [Current Scientific
              and Industrial Topics], No. 1337},
 PUBLISHER = {Hermann, Paris},
      YEAR = {1968},
     PAGES = {288 pp. (loose errata)},
   MRCLASS = {22.50 (17.00)},
  MRNUMBER = {240238},
MRREVIEWER = {G. B. Seligman},
}

@book {McKayPatera,
    AUTHOR = {McKay, W. G. and Patera, J.},
     TITLE = {Tables of dimensions, indices, and branching rules for
              representations of simple {L}ie algebras},
    SERIES = {Lecture Notes in Pure and Applied Mathematics},
    VOLUME = {69},
 PUBLISHER = {Marcel Dekker, Inc., New York},
      YEAR = {1981},
     PAGES = {v+317},
      ISBN = {0-8247-1227-7},
   MRCLASS = {17B10 (22E70 81C40)},
  MRNUMBER = {604363},
MRREVIEWER = {Bernard Kolman},
}

@article{belolipetsky2021subspace,
  title={Subspace stabilisers in hyperbolic lattices},
  author={Belolipetsky, Mikhail and Bogachev, Nikolay and Kolpakov, Alexander and Slavich, Leone},
  journal={arXiv preprint arXiv:2105.06897},
  year={2021}
}

@article{LeeOh,
  title={Orbit closures of unipotent flows for hyperbolic manifolds with Fuchsian ends},
  author={Lee, Minju and Oh, Hee},
  journal={arXiv preprint arXiv:1902.06621 and to appear Geometry and Topology},
  year={2019}
}

@article{Lowe,
  title={Rigidity of Totally Geodesic Hypersurfaces in Negative Curvature},
  author={Lowe, Ben},
  journal={arXiv preprint arXiv:2306.01254},
  year={2023}
  }

@article{ElHWilhelm,
  title={Random 3-Manifolds Have No Totally Geodesic Submanifolds},
  author={El-Hasan, Hasan and Wilhelm, Frederick},
  journal={https://arxiv.org/abs/2404.01581},
  year={2024}
}

@article {Bando,
    AUTHOR = {Bando, Shigetoshi},
     TITLE = {Real analyticity of solutions of {H}amilton's equation},
   JOURNAL = {Math. Z.},
  FJOURNAL = {Mathematische Zeitschrift},
    VOLUME = {195},
      YEAR = {1987},
    NUMBER = {1},
     PAGES = {93--97},
      ISSN = {0025-5874},
   MRCLASS = {53C20 (35B99 58G11 58G30)},
  MRNUMBER = {888130},
MRREVIEWER = {Dennis M. DeTurck},
       DOI = {10.1007/BF01161602},
}

@article {GromovSchoen,
    AUTHOR = {Gromov, Mikhail and Schoen, Richard},
     TITLE = {Harmonic maps into singular spaces and {$p$}-adic
              superrigidity for lattices in groups of rank one},
   JOURNAL = {Inst. Hautes \'{E}tudes Sci. Publ. Math.},
  FJOURNAL = {Institut des Hautes \'{E}tudes Scientifiques. Publications
              Math\'{e}matiques},
    NUMBER = {76},
      YEAR = {1992},
     PAGES = {165--246},
      ISSN = {0073-8301},
   MRCLASS = {58E20 (22E40)},
  MRNUMBER = {1215595},
MRREVIEWER = {Caio J. C. Negreiros},
       URL = {http://www.numdam.org/item?id=PMIHES_1992__76__165_0},
}

@article {Corlette,
    AUTHOR = {Corlette, Kevin},
     TITLE = {Archimedean superrigidity and hyperbolic geometry},
   JOURNAL = {Ann. of Math. (2)},
  FJOURNAL = {Annals of Mathematics. Second Series},
    VOLUME = {135},
      YEAR = {1992},
    NUMBER = {1},
     PAGES = {165--182},
      ISSN = {0003-486X},
   MRCLASS = {57S30 (22E40 53C25 57M50 58E20)},
  MRNUMBER = {1147961},
MRREVIEWER = {Christopher W. Stark},
       DOI = {10.2307/2946567},
       URL = {https://doi.org/10.2307/2946567},
}

@article{cmn22,
  title={Counting minimal surfaces in negatively curved 3-manifolds},
  author={Calegari, Danny and Marques, Fernando C and Neves, Andr{\'e}},
  journal={Duke Mathematical Journal},
  volume={171},
  number={8},
  pages={1615--1648},
  year={2022},
  publisher={Duke University Press}
}

@article {GromovThurston,
    AUTHOR = {Gromov, M. and Thurston, W.},
     TITLE = {Pinching constants for hyperbolic manifolds},
   JOURNAL = {Invent. Math.},
  FJOURNAL = {Inventiones Mathematicae},
    VOLUME = {89},
      YEAR = {1987},
    NUMBER = {1},
     PAGES = {1--12},
      ISSN = {0020-9910},
   MRCLASS = {53C20},
  MRNUMBER = {892185},
MRREVIEWER = {Karsten Grove},
       DOI = {10.1007/BF01404671},
       URL = {https://doi.org/10.1007/BF01404671},
}

@article {FJ,
    AUTHOR = {Farrell, F. T. and Jones, L. E.},
     TITLE = {Negatively curved manifolds with exotic smooth structures},
   JOURNAL = {J. Amer. Math. Soc.},
  FJOURNAL = {Journal of the American Mathematical Society},
    VOLUME = {2},
      YEAR = {1989},
    NUMBER = {4},
     PAGES = {899--908},
      ISSN = {0894-0347},
   MRCLASS = {53C20 (57R10 57R55 57R67)},
  MRNUMBER = {1002632},
MRREVIEWER = {Ronald J. Stern},
       DOI = {10.2307/1990898},
       URL = {https://doi.org/10.2307/1990898},
}

@article {LytchakPetrunin,
    AUTHOR = {Lytchak, Alexander and Petrunin, Anton},
     TITLE = {About every convex set in any generic {R}iemannian manifold},
   JOURNAL = {J. Reine Angew. Math.},
  FJOURNAL = {Journal f\"{u}r die Reine und Angewandte Mathematik. [Crelle's
              Journal]},
    VOLUME = {782},
      YEAR = {2022},
     PAGES = {235--245},
      ISSN = {0075-4102},
   MRCLASS = {53B20 (53B25 53C22)},
  MRNUMBER = {4360005},
MRREVIEWER = {Victor Bangert},
       DOI = {10.1515/crelle-2021-0058},
       URL = {https://doi.org/10.1515/crelle-2021-0058},
}

@article {lowealassal,
AUTHOR = {Al Assal, Fernando and Lowe, Ben},
TITLE= {Accumulation sets of Asymptotically Fuchsian surfaces},
JOURNAL = {Forthcoming Paper},
}

@book{BarreiraPesin2007_Nonuniform-hyperbolicity,
    Title = {Nonuniform hyperbolicity},
    Publisher = {Cambridge University Press, Cambridge},
    Year = {2007},
    Author = {Barreira, Luis and Pesin, Yakov},
    Volume = {115},
    Series = {Encyclopedia of Mathematics and its Applications},
    url = {https://doi.org/10.1017/CBO9781107326026},
    doi = {10.1017/CBO9781107326026},
    pages = {xiv+513},
}

@article{MohammadiMargulis2022_Arithmeticity-of-hyperbolic-3-manifolds-containing-infinitely-many-totally,
    Author = {Margulis, Gregorii and Mohammadi, Amir},
    Title = {Arithmeticity of hyperbolic 3-manifolds containing infinitely many totally geodesic surfaces},
    Journal = {Ergodic Theory Dynam. Systems},
    Year = {2022},
    Volume = {42},
    Number = {3},
    Pages = {1188--1219},
    url = {https://doi.org/10.1017/etds.2021.21},
    doi = {10.1017/etds.2021.21},
    pages = {1188--1219},
}

@incollection {Fisher-ICM,
    AUTHOR = {Fisher, David},
     TITLE = {Rigidity, lattices, and invariant measures beyond homogeneous
              dynamics},
 BOOKTITLE = {I{CM}---{I}nternational {C}ongress of {M}athematicians. {V}ol.
              {V}. {S}ections 9--11},
     PAGES = {3484--3507},
 PUBLISHER = {EMS Press, Berlin},
      YEAR = {[2023] \copyright 2023},
   MRCLASS = {22E40 (37A17 37C85 53C35)},
  MRNUMBER = {4680371},
}

@book {Cartan,
    AUTHOR = {Cartan, \'{E}lie},
     TITLE = {Le\c{c}ons sur la g\'{e}om\'{e}trie des espaces de {R}iemann},
      NOTE = {2d ed},
 PUBLISHER = {Gauthier-Villars, Paris},
      YEAR = {1951},
     PAGES = {viii+378},
   MRCLASS = {53.0X},
  MRNUMBER = {44878},
}

@article{knieper98,
  title={The uniqueness of the measure of maximal entropy for geodesic flows on rank 1 manifolds},
  author={Knieper, Gerhard},
  journal={Annals of mathematics},
  pages={291--314},
  year={1998},
  publisher={JSTOR}
}

@article {BaldiKlinglerUllmo,
    AUTHOR = {Baldi, Gregorio and Klingler, Bruno and Ullmo, Emmanuel},
     TITLE = {On the distribution of the {H}odge locus},
   JOURNAL = {Invent. Math.},
  FJOURNAL = {Inventiones Mathematicae},
    VOLUME = {235},
      YEAR = {2024},
    NUMBER = {2},
     PAGES = {441--487},
      ISSN = {0020-9910},
   MRCLASS = {14D07 (03C64 14C30 14G35 22F30)},
  MRNUMBER = {4689371},
       DOI = {10.1007/s00222-023-01226-0},
       URL = {https://doi.org/10.1007/s00222-023-01226-0},
}

@article {BaldiUllmo,
    AUTHOR = {Baldi, Gregorio and Ullmo, Emmanuel},
     TITLE = {Special subvarieties of non-arithmetic ball quotients and
              {H}odge theory},
   JOURNAL = {Ann. of Math. (2)},
  FJOURNAL = {Annals of Mathematics. Second Series},
    VOLUME = {197},
      YEAR = {2023},
    NUMBER = {1},
     PAGES = {159--220},
      ISSN = {0003-486X},
   MRCLASS = {14G35 (03C64 14C30 14P10 22E40 32G20)},
  MRNUMBER = {4513144},
MRREVIEWER = {Rolf Berndt},
       DOI = {10.4007/annals.2023.197.1.3},
       URL = {https://doi.org/10.4007/annals.2023.197.1.3},
}

@article {BFMS2,
    AUTHOR = {Bader, Uri and Fisher, David and Miller, Nicholas and Stover,
              Matthew},
     TITLE = {Arithmeticity, superrigidity and totally geodesic submanifolds
              of complex hyperbolic manifolds},
   JOURNAL = {Invent. Math.},
  FJOURNAL = {Inventiones Mathematicae},
    VOLUME = {233},
      YEAR = {2023},
    NUMBER = {1},
     PAGES = {169--222},
      ISSN = {0020-9910},
   MRCLASS = {53C20 (22E40 32Q45 53C40)},
  MRNUMBER = {4601998},
       DOI = {10.1007/s00222-023-01186-5},
       URL = {https://doi.org/10.1007/s00222-023-01186-5},
}

@article{BFMS1,
    Author = {Bader, Uri and Fisher, David and Miller, Nicholas and Stover, Matthew},
    Title = {Arithmeticity, superrigidity, and totally geodesic submanifolds},
    Journal = {Ann. of Math. (2)},
    Year = {2021},
    Volume = {193},
    Number = {3},
    Pages = {837--861},
    url = {https://doi.org/10.4007/annals.2021.193.3.4},
    doi = {10.4007/annals.2021.193.3.4},
    pages = {837--861},
}

@article{CekicLefeuvreMoroianu2023_On-the-ergodicity-of-the-frame-flow-on-even-dimensional-manifolds,
    Title = {On the ergodicity of the frame flow on even-dimensional manifolds},
    Author = {Cekic, Mihajlo and Lefeuvre, Thibault and Moroianu, Andrei and Semmelmann, Uwe},
    Year = {2023},
    doi = {10.48550/arXiv.2111.14811},
}

@article{BierstoneMilman1988_Semianalytic-and-subanalytic-sets,
    Author = {Bierstone, Edward and Milman, Pierre D.},
    Title = {Semianalytic and subanalytic sets},
    Journal = {Inst. Hautes \'{E}tudes Sci. Publ. Math.},
    Year = {1988},
    Number = {67},
    Pages = {5--42},
    url = {http://www.numdam.org/item?id=PMIHES_1988__67__5_0},
    pages = {5--42},
}

@book{ONeill1983_Semi-Riemannian-geometry,
    Title = {Semi-{R}iemannian geometry},
    Publisher = {Academic Press, Inc. [Harcourt Brace Jovanovich, Publishers], New York},
    Year = {1983},
    Author = {O'Neill, Barrett},
    Volume = {103},
    Series = {Pure and Applied Mathematics},
    pages = {xiii+468},
}

@article {heintze1974homogeneous,
    AUTHOR = {Heintze, Ernst},
     TITLE = {On homogeneous manifolds of negative curvature},
   JOURNAL = {Math. Ann.},
  FJOURNAL = {Mathematische Annalen},
    VOLUME = {211},
      YEAR = {1974},
     PAGES = {23--34},
      ISSN = {0025-5831,1432-1807},
   MRCLASS = {53C30 (22E25)},
  MRNUMBER = {353210},
MRREVIEWER = {K.\ Nomizu},
       DOI = {10.1007/BF01344139},
       URL = {https://doi.org/10.1007/BF01344139},
}

@book{Narasimhan1966_Introduction-to-the-theory-of-analytic-spaces,
    Title = {Introduction to the theory of analytic spaces},
    Publisher = {Springer-Verlag, Berlin-New York},
    Year = {1966},
    Author = {Narasimhan, Raghavan},
    Volume = {No. 25},
    Series = {Lecture Notes in Mathematics},
    pages = {iii+143},
}

@article{ls17,
  title={Manifolds with many hyperbolic planes},
  author={Lin, Samuel and Schmidt, Benjamin},
  journal={Differential Geometry and its Applications},
  volume={52},
  pages={121--126},
  year={2017},
  publisher={Elsevier}
}

@book{Dajczer1990_Submanifolds-and-isometric-immersions,
    Title = {Submanifolds and isometric immersions},
    Publisher = {Publish or Perish, Inc., Houston, TX},
    Year = {1990},
    Author = {Dajczer, Marcos},
    Volume = {13},
    Series = {Mathematics Lecture Series},
    pages = {x+173},
}

@book {Spivak,
    AUTHOR = {Spivak, Michael},
     TITLE = {A comprehensive introduction to differential geometry. {V}ol.
              {III}},
   EDITION = {Second},
 PUBLISHER = {Publish or Perish, Inc., Wilmington, DE},
      YEAR = {1979},
     PAGES = {xii+466},
      ISBN = {0-914098-83-7},
   MRCLASS = {53-01},
  MRNUMBER = {532832},
MRREVIEWER = {M. A. Akivis},
}

@article {MurphyWilhelm,
    AUTHOR = {Murphy, Thomas and Wilhelm, Frederick},
     TITLE = {Random manifolds have no totally geodesic submanifolds},
   JOURNAL = {Michigan Math. J.},
  FJOURNAL = {Michigan Mathematical Journal},
    VOLUME = {68},
      YEAR = {2019},
    NUMBER = {2},
     PAGES = {323--335},
      ISSN = {0026-2285},
   MRCLASS = {53C20 (53A99 53C40)},
  MRNUMBER = {3961219},
MRREVIEWER = {Yun Tao Zhang},
       DOI = {10.1307/mmj/1555034652},
       URL = {https://doi.org/10.1307/mmj/1555034652},
}

@article {survey_tot_geod_submanifolds,
    AUTHOR = {Van Lindt, Dirk and Verstraelen, Leopold},
     TITLE = {A survey on axioms of submanifolds in {R}iemannian and
              {K}aehlerian geometry},
   JOURNAL = {Colloq. Math.},
  FJOURNAL = {Colloquium Mathematicum},
    VOLUME = {54},
      YEAR = {1987},
    NUMBER = {2},
     PAGES = {193--213},
      ISSN = {0010-1354,1730-6302},
   MRCLASS = {53C40 (53C55)},
  MRNUMBER = {948513},
MRREVIEWER = {Y.\ Tashiro},
       DOI = {10.4064/cm-54-2-193-213},
       URL = {https://doi.org/10.4064/cm-54-2-193-213},
}

@book{krantz2002primer,
  title={A primer of real analytic functions},
  author={Krantz, Steven G and Parks, Harold R},
  year={2002},
  publisher={Springer Science \& Business Media}
}

\end{document}